\documentclass{amsart}

\usepackage{amssymb,amsmath, amsthm, amsfonts,comment}
\usepackage{bbm}
\usepackage{mathrsfs}
\usepackage[ruled,linesnumbered]{algorithm2e}

\usepackage[margin=1in]{geometry}
\usepackage{graphicx, subcaption}
\usepackage{xcolor}
\usepackage{mathabx} 
\usepackage{bbm}    
\usepackage{tikz-cd,pgfplots}

\usepackage[doi=false,isbn=false,url=false,style=alphabetic,maxalphanames=99, maxbibnames=99]{biblatex}
\def\spann{\mathop{\rm span}\nolimits}

\def\tr{\mathop{\rm trace}}

\DeclareMathOperator{\fold}{fold}
\def\diag{\mathop{\rm diag}}
\def\bdiag{\mathop{\rm bdiag}}

\renewcommand{\vec}{\mathop{\rm{vec}}}

\newcommand{\bbN}{\mathbb{N}}

\newcommand{\bbR}{\mathbb{R}}
\newcommand{\bbC}{\mathbb{C}}

\newcommand{\bbF}{\mathbb{F}}

\newcommand{\bbS}{\mathbb{S}}

\newcommand{\cA}{\mathcal{A}}
\newcommand{\cB}{\mathcal{B}}
\newcommand{\cC}{\mathcal{C}}
\newcommand{\cD}{\mathcal{D}}

\newcommand{\cU}{\mathcal{U}}
\newcommand{\cV}{\mathcal{V}}

\newcommand{\cX}{\mathcal{X}}
\newcommand{\cY}{\mathcal{Y}}
\newcommand{\cS}{\mathcal{S}}

\newcommand{\cQ}{\mathcal{Q}}
\newcommand{\cI}{\mathcal{I}}

\newcommand{\cW}{\mathcal{W}}

\newcommand{\ba}{\mathbf{a}}
\newcommand{\bb}{\mathbf{b}}

\newcommand{\bq}{\mathbf{q}}

\newcommand{\bx}{\mathbf{x}}

\usepackage{thmtools, thm-restate}
\declaretheorem[numberwithin=section]{theorem}

\newtheorem{corollary}[theorem]{Corollary}
\newtheorem{lemma}[theorem]{Lemma}
\newtheorem{proposition}[theorem]{Proposition}
\newtheorem{defn}[theorem]{Definition}
\theoremstyle{definition}
\newtheorem{example}{Example}[section]

\newtheorem{remark}{Remark}[section]

\newcommand{\starM}{\star_{M}}

\usepackage{hyperref}
\hypersetup{
    colorlinks=true,
    linkcolor = blue,
    citecolor = blue,
    urlcolor = blue
}

\usepackage{cleveref}

\def\mydefb#1{\expandafter\def\csname bf#1\endcsname{\mathbf{#1}}}
\def\mydefallb#1{\ifx#1\mydefallb\else\mydefb#1\expandafter\mydefallb\fi}
\mydefallb aAbBcCdDeEfFgGhHiIjJkKlLmMnNoOpPqQrRsStTuUvVwWxXyYzZ\mydefallb

\def\mydefb#1{\expandafter\def\csname #1bb\endcsname{\mathbb{#1}}}
\def\mydefallb#1{\ifx#1\mydefallb\else\mydefb#1\expandafter\mydefallb\fi}
\mydefallb aAbBcCdDeEfFgGhHiIjJkKlLmMnNoOpPqQrRsStTuUvVwWxXyYzZ\mydefallb

\def\mydefb#1{\expandafter\def\csname #1cal\endcsname{\mathcal{#1}}}
\def\mydefallb#1{\ifx#1\mydefallb\else\mydefb#1\expandafter\mydefallb\fi}
\mydefallb aAbBcCdDeEfFgGhHiIjJkKlLmMnNoOpPqQrRsStTuUvVwWxXyYzZ\mydefallb

\def\mydefgreek#1{\expandafter\def\csname bf#1\endcsname{\text{\boldmath$\mathbf{\csname #1\endcsname}$}}}
\def\mydefallgreek#1{\ifx\mydefallgreek#1\else\mydefgreek{#1}%
   \lowercase{\mydefgreek{#1}}\expandafter\mydefallgreek\fi}
\mydefallgreek {alpha}{Alpha}{beta}{Beta}{gamma}{Gamma}{delta}{Delta}{epsilon}{Epsilon}{zeta}{Zeta}{eta}{Eta}{theta}{Theta}{iota}{Iota}{kappa}{Kappa}{lambda}{Lambda}{mu}{Mu}{nu}{Nu}{omicron}{Omicron}{pi}{Pi}{rho}{Rho}{sigma}{Sigma}{tau}{Tau}{upsilon}{Upsilon}{phi}{Phi}{xi}{Xi}{chi}{Chi}{psi}{Psi}{omega}{Omega}\mydefallgreek

\newcommand{\Matlab}{{\sc Matlab}}

\DeclareMathOperator{\subjectto}{s.t.}
\DeclareMathOperator{\tube}{tube}
\DeclareMathOperator{\mycirc}{circ}
\DeclareMathOperator{\myvec}{vec}

\DeclareMathOperator{\mytrace}{trace}
\DeclareMathOperator{\mydet}{det}
\DeclareMathOperator{\GL}{GL}
\DeclareMathOperator{\Dgroup}{D}

\DeclareMathOperator{\Ogroup}{O}
\DeclareMathOperator{\myPSD}{PSD}
\DeclareMathOperator{\mymat}{mat}

\DeclareMathOperator{\starMrank}{\starM-rank}
\DeclareMathOperator{\sgn}{sgn}

\definecolor{EmoryBlue}{RGB}{1, 33, 105} 
\definecolor{EmoryDarkBlue}{RGB}{12, 35, 64} 
\definecolor{EmoryMediumBlue}{RGB}{0, 51, 160} 
\definecolor{EmoryLightBlue}{RGB}{0, 125, 186} 
\definecolor{EmoryYellow}{RGB}{242, 169, 0} 
\definecolor{EmoryGold}{RGB}{181, 133, 0} 
\definecolor{EmoryMetallicGold}{RGB}{132, 117, 78} 

\definecolor{TuftsBlue}{RGB}{49,114,174}
\definecolor{TuftsBrown}{RGB}{94,75,60}

\definecolor{mycolor0}{rgb}{1, 1, 1}%
\definecolor{mycolor1}{rgb}{0.00000,0.44700,0.74100}%
\definecolor{mycolor2}{rgb}{0.85000,0.32500,0.09800}%
\definecolor{mycolor3}{rgb}{0.92900,0.69400,0.12500}%
\definecolor{mycolor4}{rgb}{0.49400,0.18400,0.55600}%
\definecolor{mycolor5}{rgb}{0.46600,0.67400,0.18800}%
\definecolor{mycolor6}{rgb}{0.30100,0.74500,0.93300}%

\definecolor{jet1}{rgb}{0.0000,0.0000,0.6667}
\definecolor{jet2}{rgb}{0.0000,0.0000,1.0000}
\definecolor{jet3}{rgb}{0.0000,0.3333,1.0000}
\definecolor{jet4}{rgb}{0.0000,0.6667,1.0000}
\definecolor{jet5}{rgb}{0.0000,1.0000,1.0000}
\definecolor{jet6}{rgb}{0.3333,1.0000,0.6667}
\definecolor{jet7}{rgb}{0.6667,1.0000,0.3333}
\definecolor{jet8}{rgb}{1.0000,1.0000,0.0000}
\definecolor{jet9}{rgb}{1.0000,0.6667,0.0000}
\definecolor{jet10}{rgb}{1.0000,0.3333,0.0000}
\definecolor{jet11}{rgb}{1.0000,0.0000,0.0000}

\usepackage[most]{tcolorbox}
\usepackage{cleveref}

\makeatletter
\crefformat{tcb@cnt@mytheo}{theorem~#2#1#3}
\Crefformat{tcb@cnt@mytheo}{Theorem~#2#1#3}
\crefformat{tcb@cnt@mylemma}{lemma~#2#1#3}
\Crefformat{tcb@cnt@mylemma}{Lemma~#2#1#3}
\crefformat{tcb@cnt@myproto}{prototype problem~#2#1#3}
\Crefformat{tcb@cnt@myproto}{Prototype Problem~#2#1#3}
\crefformat{tcb@cnt@myaction}{action items~#2#1#3}
\Crefformat{tcb@cnt@myaction}{Action Items~#2#1#3}
\makeatother

\newtcbtheorem[auto counter, number within=section]{mytheo}{Theorem}%
{colback=EmoryBlue!5,colframe=EmoryBlue, fonttitle=\bfseries}{thm}

\newtcbtheorem[auto counter, number within=section]{mylemma}{Lemma}%
{colback=gray!5,colframe=gray, fonttitle=\bfseries}{lem}

\newtcbtheorem[auto counter, number within=section]{myproto}{Problem}%
{colback=EmoryGold!5,colframe=EmoryGold, fonttitle=\bfseries}{proto}

\newtcbtheorem[auto counter, number within=section]{myaction}{Action Iterms}%
{colback=gray!5,colframe=gray, fonttitle=\bfseries}{action}

\title{Tensor-Tensor Products, Group Representations, and Semidefinite Programming}
\author{Alex Dunbar$^{*,1,2}$ \and Elizabeth Newman$^{\dagger,1,3}$}
\date{%
    \today\\
    $^*$\href{mailto:alex.dunbar@emory.edu}{alex.dunbar@emory.edu}\\
    $^\dagger$\href{mailto:elizabeth.newman@emory.edu}{elizabeth.newman@emory.edu}, \href{mailto:e.newman@tufts.edu}{e.newman@tufts.edu}, \\
    $^1$Emory University, Atlanta, GA 30322\\%
    $^2$Georgia Institute of Technology, Atlanta, GA 30332\\
    $^3$Tufts University, Medford, MA, 02155
}

\begin{document}

\begin{abstract}
The $\starM$-family of tensor-tensor products is a framework which generalizes many properties from linear algebra to third order tensors. Here, we investigate positive semidefiniteness and semidefinite programming under the $\starM$-product. Critical to our investigation is a connection between the choice of matrix $M$ in the $\starM$-product and the representation theory of an underlying group action. Using this framework, third order tensors equipped with the $\starM$-product are a natural setting for the study of invariant semidefinite programs. As applications of the $M$-SDP framework, we provide a characterization of certain nonnegative quadratic forms and solve low-rank tensor completion problems.
\end{abstract}

\maketitle

\section{Introduction}\label{sec:introduction}

Multiway operators or \emph{tensors} are central objects in many fields of pure and applied mathematics, including data science applications~\cite{goos_multilinear_2002, frolov_tensor_2017}, multidimensional scientific simulations~\cite{ballester-ripoll_tthresh_2020}, and tensor fields (i.e., multiway-array-valued functions) for Riemannian manifolds~\cite{dodson_tensor_1991} and mechanical stress tensors~\cite{irgens_continuum_2008}. 
The breadth of multilinear applications has generated numerous factorization strategies to extract multidimensional features and operate in high dimensions; see, e.g.,~\cite{ballard_tensor_2025, kolda_tensor_2009} for comprehensive reviews.  
Three foundational tensor factorizations that emulate properties of the matrix singular value decomposition (SVD)~\cite{golub_matrix_2013} are Canonical Polyadic (CP)~\cite{hitchcock_expression_1927, harshman_foundations_1970, carroll_analysis_1970}, higher-order SVD (HOSVD)~\cite{tucker_mathematical_1966, de_lathauwer_multilinear_2000}, and tensor train~\cite{oseledets_tensor-train_2011}. 
Each tensor decomposition brings distinct advantages and algebraic consequences. 
The CP format offers uniqueness and interpretability of the extracted features, but common linear algebra computations become NP-hard (e.g., determining rank and eigenvalues) \cite{hillar_most_2013}. 
The HOSVD excels at data compression, but lacks Schmidt-Eckart-Young-Mirsky-like guarantees for optimal low-rank approximations~\cite{schmidt_zur_1907, eckart_approximation_1936, mirsky_symmetric_1960}. 
The tensor train format is designed to approximate high-dimensional operators arising in quantum chemistry~\cite{orus_practical_2014} and high-dimensional partial differential equations~\cite{bachmayr_low-rank_2023}, but is often viewed as a powerful and strategic computational tool rather than analyzed from a linear algebraic perspective. 
The algebraic sacrifices of these classical tensor decompositions can make revealing underlying symmetries, invariants, and other geometric structures challenging in high dimensions and potentially eliminate the opportunity to exploit such structures for efficient algorithm development.

Recent work has introduced a family of tensor operators, called the $\starM$-product (the prefix is pronounced ``star-M''), that extends traditional matrix-matrix multiplication to multidimensional arrays~\cite{kernfeld_tensortensor_2015, kilmer_third-order_2013, kilmer_factorization_2011}.  
This tensor operation gives rise to provably optimal multiway data representations~\cite{kilmer_tensor-tensor_2021} and has demonstrated empirical success across applications in imaging~\cite{newman_nonnegative_2020, soltani_tensor-based_2016}, tensor completion~\cite{zhang_exact_2017, kong_tensor_2021}, and machine learning~\cite{newman_stable_2024}.   
The key perspective of the $\starM$-product is to consider an order-three tensor to be a two-dimensional matrix with vector-valued entries. 
The $\starM$-product thereby multiplies two tensors by replacing the underlying notion of scalar multiplication with a vector-vector product defined by an invertible matrix, $M$. 
The original $t$-product~\cite{kilmer_factorization_2011} considered $M$ as a discrete Fourier transform matrix, giving rise to analysis of numerical algorithms under the algebra of circulants~\cite{kilmer_third-order_2013, gleich_power_2013}. 
Connections between the $\starM$-product and other tensor factorization strategies were introduced in~\cite{kilmer_tensor-tensor_2021}. 

A strength of the $\starM$-product framework is that it allows for tensor properties and tensor decompositions which are analogous to familiar matrix decompositions in numerical linear algebra, such as the existence of a singular value decomposition ~\cite{kilmer_factorization_2011, kilmer_tensor-tensor_2021}. One matrix property which is widely applicable in both pure and applied mathematics is positive semidefiniteness. In convex optimization, positive semidefiniteness is used to define \emph{semidefinite programming problems}--optimization problems whose decision variables are positive semidefinite matrices. Semidefinite programming is well-studied; see e.g., \cite{vandenberghe_semidefinite_1996}, and has applications in polynomial optimization \cite{lasserre_global_2001, parrilo_semidefinite_2003}, relaxations of integer programs \cite{goemans_improved_1995}, and applied algebraic geometry \cite{blekherman_semidefinite_2012} among others. Semidefiniteness has been studied in the case of the original $t$-product \cite{kilmer_factorization_2011} in \cite{zheng_t-positive_2021, zheng_unconstrained_2022, marumo_t-semidefinite_2024}, but not for the general $\starM$-product. 

\subsection{Our Contributions}

In this paper, we develop the theory of $\starM$-semidefinite tensors and corresponding optimization problems. As is common in the $\starM$-product literature, the condition of positive semidefiniteness can be checked on slices in the transform domain; that is, the tensor obtained by changing basis on the space of tubes according to the matrix $M$. Viewing this as a block diagonalization result, we make a connection to group invariant semidefinite programs, which requires a representation-theoretic perspective on the $\starM$-product. Thus far, the mathematics behind the $\starM$-product has been primarily discussed in the language of matrix algebras, which has enabled elegant extensions of numerical linear algebra tools and algorithms to multilinear settings. 
Our primary achievement is providing the first unified view of the $\starM$-product from algebraic and geometric perspectives which gives interpretability to the matrix $M$. 
Notably, we completely characterize the geometric space of tensors which represent group equivariant transformations under the $\starM$-product for a given matrix $M$. 
Group equivariance, the property that a group action on a function input is equivalent to a group action on a function output, is a symmetry-preserving property that has wide utility in statistical inference~\cite{lehmann_theory_1998}, differential equations~\cite{olver_group-invariant_1987},  and machine learning~\cite{bronstein_geometric_2021, lim_what_2023, celledoni_equivariant_2021, maile_equivariance-aware_2022}. 
With a group representation interpretation, we are able to exploit group-theoretic tools on tensor-tensor products. 

We briefly highlight key results presented in our paper. 

\begin{itemize}
    
    \item We introduce $\starM$-semidefinite programming, the first use of a general $\starM$-product for this application. 
    In the process, we provide rigorous theoretical analysis of $\starM$-semidefinite tensors, extending numerous classic equivalent definitions to the tensor case. 

    \item We connect $\starM$-semidefinite programming to group invariant semidefinite programs. In the process, we present new algebraic and geometric insight into the $\starM$-product from a representation theoretic perspective. Specifically, we highlight the importance of the choice of matrix $M$ via an interpretation through Schur's Lemma.
    
    \item To complement the new theoretical advancements and demonstrate the practicality of our approach, we provide numerical experiments of solving $\starM$-semidefinite programs on both video data and hyperspectral data.  
    Code to reproduce all experiments is provided publicly at \url{https://github.com/elizabethnewman/tsdp}. 

\end{itemize}

\subsection{Related Work}

The most closely related work is the series of articles \cite{zheng_t-positive_2021, zheng_unconstrained_2022}, where the authors study semidefinite tensors with the $t$-product and apply such tensors to the study of polynomial optimization problems. An alternative (equivalent) definition of $t$-positive semidefinite tensors was given in \cite[Definition 2.13]{kilmer_third-order_2013} where such tensors were used for numerical algorithms. The use of  $t$-semidefinite programming was extended to the constrained polynomial optimization case in \cite{marumo_t-semidefinite_2024}. The study of $t$-PSD tensors has been further developed in  \cite{ju_geometric_2024}, which studies the geometry of the set of $t$-PSD tensors, and \cite{qi_t-quadratic_2021}, which studies quadratic forms. 

The theory of invariant semidefinite programs and their application to sums of squares polynomials was initiated in \cite{gatermann_symmetry_2004}. Theoretical and computational extensions of this framework have since been substaintially developed; see for example \cite{de_klerk_exploiting_2010, lofberg_pre-_2009, anjos_invariant_2012, riener_exploiting_2013}. 

Nuclear-norm minimization approaches to tensor completion which utilize tensor-tensor products have been studied in \cite{semerci_tensor-based_2014, zhang_corrected_2019, qi_tensor_2017, zhou_tensor_2018, zhang_exact_2017}. A semidefinite programing formulation of the nuclear norm problem in the matrix case was developed in \cite{recht_guaranteed_2010}.

Finally, representation theoretic ideas have been used in numerical multilinear algebra in ~\cite{sprangers_group-invariant_2023}, where the authors study tensor-train decompositions for tensors which represent group invariant multilinear maps. Representation theoretic tools have also seen increased attention in computational and data sciences; see e.g., \cite{bronstein_geometric_2021} for an overview. 

\subsection{Organization} The remainder of the paper is organized as follows. Section \ref{sec:notation} fixes notation and provides preliminary results for tensors and tensor-tensor products. In Section \ref{sec:SemidefiniteTensors}, we develop a theory of $\starM$-PSD tensors and $M$-semidefinite programs. In Section \ref{sec:Equivariance}, we prove group equivariance properties for the $\starM$-product and connect $\starM$-PSD tensors to invariant SDPs. Sections \ref{sec:MSOS} and \ref{sec:TensorCompletion} provide applications to sums of squares polynomials and low-rank tensor completion problems, respectively. We conclude in Section \ref{sec:Completion_Experiments} with numerical experiments.   

\section{Notation and Background}\label{sec:notation}

This paper focuses on order-$3$ tensors, which we denote with calligraphic letters, e.g., $\cA \in \bbF^{n_1\times n_2\times n_3}$ where the field $\bbF$ is either the real numbers $\bbR$ or the complex numbers $\bbC$. 
We use \Matlab\ indexing notation to select different parts of a tensor; e.g., $\cA_{:,:,k}$ or $\cA(:,:,k)$ is the $k$-th frontal slice of the tensor, a matrix of size $n_1\times n_2$.   
We call tensors of form $1 \times 1 \times n_3$ for some $n_3 \geq 2$ \emph{tubes} and denote tubes with lowercase bold letters, e.g., $\bfa \in \Fbb^{1\times 1 \times n_3}$.  
From this terminology, we can view a tensor $\cA \in \bbF^{n_1\times n_2\times n_3}$ as an $n_1\times n_2$ matrix where each entry is a tube; that is, $\cA_{i,j,:} = \bfa_{ij} \in \Fbb^{1\times 1\times n_3}$ for $i\in [n_1]$ and $j\in [n_2]$. 
The bracket notation is defined to be $[m] = \{1,2,\ldots, m\}$ for any positive integer $m$. 
See~\Cref{fig:tensor_notation} for an illustration of the tensor notation.

\begin{figure}
    \centering
    \begin{subfigure}{0.4\linewidth}
    \centering
    \includegraphics[width=0.7\linewidth]{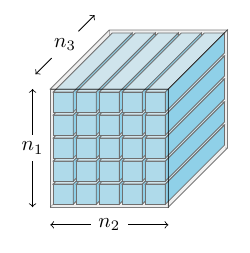}
    \subcaption{Matrix of tubes $\cA\in \Fbb_{M}^{n_1\times n_2}$}
    \end{subfigure}
    \begin{subfigure}{0.4\linewidth}
    \centering
    \includegraphics[width=0.7\linewidth]{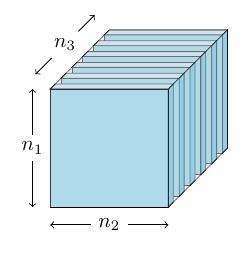}
    \subcaption{Frontal slices  $\cA_{:,:,k}\in \Fbb^{n_1\times n_2}$ for $k\in [n_3]$}
    \end{subfigure}
    
    \caption{Illustration of tensor notation for $M\in \GL_{n_3}(\bbF)$.}
    \label{fig:tensor_notation}
\end{figure}

\subsubsection{Vectorization} To avoid ambiguity with tubes, we denote column vectors with non-bold lowercase letters; e.g., $a\in \Fbb^{n_3}$. 
To map between tubes and vectors, we define two isomorphisms, $\tube:\bbF^{n_3} \to \bbF^{1\times 1 \times n_3}$ and $\vec:\bbF^{1\times 1 \times n_3} \to \bbF^{n_3}$, defined in coordinates by 
\begin{align}\label{eq:tubevec}
    \tube(x)_{1,1,k} = x_k \qquad \text{and} \qquad \vec(\bx)_k = \bx_{1,1,k} \qquad \text{for } k\in [n_3].
\end{align}
For ease of notation, we will sometimes write $\bx \equiv x$ if $\bx = \tube(x)$. 

\subsubsection{Tubal Transformations} To apply linear transformations along the tubes of a tensor, we use the mode-$3$ product, $\times_3: \Fbb^{n_1\times n_2\times n_3} \times \Fbb^{p\times n_3} \to \Fbb^{n_1\times n_2 \times p}$, denoted and defined as
    \begin{align}\label{eq:mode3}
        (\cA \times_3 M)_{i,j,:} = \tube(M\vec(\bfa_{ij}))
    \end{align}
for $i\in [n_1]$ and $j\in [n_2]$. 
When the matrix $M$ is clear from context, we denote $\widehat{\cA} = \cA \times_3 M$ and refer to $\widehat{\cA}$ as the image of the tensor $\cA$ in the \emph{transform domain}.  

\subsubsection{Facewise Operations} We often will discuss the structure of a tensor and implement algorithms facewise (i.e., by properties or actions of the frontal slices). 
A tensor $\cD\in \Fbb^{n_1 \times n_2 \times n_3}$ is \emph{facewise diagonal} or \emph{f-diagonal} if each frontal slice is a diagonal matrix; i.e., for each $k \in [n_3]$, the matrix $\cD_{:,:,k}\in \Dgroup_{n_1,n_2}(\Fbb)$ where $\Dgroup_{n_1,n_2}(\Fbb)$ denotes the set of $n_1\times n_2$ diagonal matrices over $\Fbb$. 
If $n = n_1 = n_2$, we write the set of diagonal matrices as $\Dgroup_n(\Fbb)$. 
The facewise identity tensor $\cI\in \Fbb^{n\times n\times n_3}$ is an f-diagonal tensor where each frontal slice is the $n\times n$ identity matrix; i.e., $\cI_{:,:,k} = I_n$ for all $k\in [n_3]$. 
We can multiply two tensors facewise by multiplying corresponding frontal slices. 
Specifically, the \emph{facewise product}, $\smalltriangleup: \Fbb^{n_1\times m\times n_3} \times \Fbb^{m\times n_2\times n_3} \to \Fbb^{n_1\times n_2\times n_3}$, computes $\cC = \cA \smalltriangleup \cB$ where $\cC_{:,:,k} = \cA_{:,:,k}\cB_{:,:,k}$ for $k\in [n_3]$. 
Computationally, the facewise product can be parallelized over the frontal slices. In terms of linear algebra, we can phrase the facewise product in terms of block diagonal matrices. Specifically, for a tensor $\cA \in \Fbb^{n_1\times n_2\times n_3}$, define its \emph{block-diagonalization} as

\begin{equation}\label{eq:BlockDiag}
\bdiag(\cA) = \begin{bmatrix} \cA_{:,:,1} & & & \\ & \cA_{:,:,2} & & \\ & & \ddots & \\ & & & \cA_{:,:,n_3}\end{bmatrix} \in \Fbb^{n_1n_3\times n_2n_3}
\end{equation}
In this notation, we have that $\bdiag(\cA \smalltriangleup \cB) = \bdiag(\cA)\bdiag(\cB)$.

\subsection{An Algebraic Ring Perspective of the Tubal $\starM$-Product} 

Following the works~\cite{kernfeld_tensortensor_2015, kilmer_factorization_2011}, we provide a brief overview of core properties of the $\starM$-framework, starting from the foundation of algebraically-motivated tensor-tensor products (the $\starM$-product), subsequent linear algebraic properties induced by the $\starM$-product, and concluding with a $\starM$-analog of the singular value decomposition (SVD).  

Fundamentally, the $\starM$-product defines a multiplication rule for third-order tensors that is parameterized by an invertible matrix $M$; that is, $M\in \GL_{n_3}(\Fbb)$ where $\GL_{n_3}(\Fbb)$ is the general linear group over $\Fbb$ or the set of invertible $n_3\times n_3$ matrices with entries in $\Fbb$. 
In some cases, we will restrict ourselves to unitary or orthogonal matrices $M$ with $M^H M = I_{n_3}$.  
We denote as $M\in \Ogroup_{n_3}(\Fbb)$ where $\Ogroup_{n_3}(\Fbb)$ is the group of $n_3\times n_3$ unitary or orthogonal matrices, when $\Fbb = \bbC$ or $\bbR$, respectively. 

We start by defining the $\starM$-product on tubes. 
Let $\ba,\bb \in \bbF^{1\times 1 \times n_3}$ and fix $M \in \GL_{n_3}(\bbF)$. 
The operation $\starM: \bbF^{1\times 1 \times n_3} \times \bbF^{1\times 1 \times n_3} \to \bbF^{1\times 1 \times n_3}$ is defined as
	\begin{align}\label{eq:starm_tube}
	\ba \starM \bb = \left((\ba \times_3 M) \odot (\bb \times_3 M)\right)\times_3M^{-1}
	\end{align}
where $\odot$ is the Hadamard pointwise product.

From the isomorphisms in~\eqref{eq:tubevec}, an equivalent vectorized presentation of the tubal $\starM$-product in~\eqref{eq:starm_tube} is
	\begin{align}\label{eq:starm_tube_vec}
	\ba \starM \bb = \tube(M^{-1}\diag(M\vec(\ba))M\vec(\bb))
	\end{align}
where $\diag: \Fbb^{n_3} \to \Dgroup_{n_3}(\Fbb)$ is a bijection that maps the entries of a vector  to the diagonal entries of a diagonal matrix.

 Note that multiplictaion by a tube $\ba \in \Fbb^{1\times 1\times n_3}$ is an $\Fbb$-linear transformation. We denote 
 this transformation $T_{\ba}: \Fbb^{1\times 1\times n_3} \to \Fbb^{1\times 1\times n_3}$; that is, $T_{\ba}(\bfx) = \bfa \starM \bfx$ for $\bfx\in \Fbb^{1\times 1\times n_3}$. 
Under the identification in~\eqref{eq:starm_tube_vec}, $T_{\ba}$ has a matrix representative 
\begin{align}\label{eq:matrix_representative}
	M^{-1}\diag(M\vec(\ba))M = \sum_{i = 1}^{n_3}a_iM^{-1}\diag(M(:,i))M. 
\end{align}
In other words, $T_{\bfa}(\bfx) \equiv A\vec(\bx)$ where $A$ is the matrix in~\eqref{eq:matrix_representative}. 
We provide a concrete example of the $\starM$-algebraic structure in~\Cref{ex:t_prod}. 

\begin{example}[The matrix representations of the $t$-product \cite{kilmer_factorization_2011}]\label{ex:t_prod}
The $t$-product is the $\starM$-product with $M$ taken to be the (unnormalized) discrete Fourier matrix $F\in \bbC^{n_3\times n_3}$. 
Each entry of $F$ is a power of a complex root of unity; that is, $F_{j+1,k+1} = e^{2\pi \mathbbm{i} jk/n_3}$ for $j,k=0,\ldots, n_3-1$ where $\mathbbm{i} = \sqrt{-1}$. 
The matrix representative of the linear transformation $T_\ba: \bbC^{1\times 1\times n_3} \to \bbC^{1\times 1\times n_3}$ is the circulant matrix
\begin{align}\label{eq:tprod_tubal_representation}
   F^{-1}\diag(F\vec(\ba))F = \begin{bmatrix} a_1 & a_{n_3} & \dots & a_2\\
a_2 & a_1 & \dots & a_3\\
\vdots & & \ddots & \vdots\\
a_{n_3} & a_{n_3-1} & \dots & a_1\end{bmatrix}
= \sum_{i=1}^{n_3}a_i Z^{i-1}
\end{align}
where $Z$ is the circulant downshift matrix  with the property $Z^{n_3} = I$; i.e.,
	\begin{align}\label{eq:circulant_downshift}
	Z = \begin{bmatrix}
	0 & 0 & \cdots  & 1\\
	1 & 0 & \cdots & 0 \\
	0 & \ddots  & \ddots & \vdots\\
	0 & 0 & 1 & 0
	\end{bmatrix}.
	\end{align}

\end{example}
As a result of the matrix representation of tubal multiplication, the $\starM$-product gives $\bbF^{1\times 1 \times n_3}$ the structure of a \emph{$\bbF$-algebra with multiplicative identity} $\bfe_M = \tube(M^{-1}\mathbbm{1}_{n_3})$, where $\mathbbm{1}_{n_3}$ is the $n_3\times 1$ all ones vector~\cite{kernfeld_tensortensor_2015}. 
We denote the ring of tubes given by the $\starM$-product over the field $\bbF$ by $\bbF_M$.  In what follows, it will sometimes be advantageous to ``forget" the ring structure on $\bbF_M$, and only emphasize the $\bbF$-vector space strucutre, in which case we use the notation $\bbF^{1\times 1\times n_3}$. 
We note that the rings $\bbF_M$ for $M\in \GL_{n_3}(\bbF)$ are all isomorphic to one another in the category of $\bbF$-algebras.

\begin{proposition}\label{prop:All_iso}
    Let $M \in \GL_{n_3}(\bbF)$ and $I$ be the $n_3\times n_3$ identity matrix. Then, $\bbF_M \simeq \bbF_I$ as $\bbF$-algebras. 
\end{proposition}

\begin{proof}
The map $\phi:\bbF_M \to \bbF_I$ defined by $\phi(\ba) = \ba \times_3 M$ is bijective and $\bbF$-linear since $M$ is invertible. Following~\eqref{eq:starm_tube_vec}, it is a $\bbF$-algebra homomorphism because 
\begin{subequations}
\begin{align}
    \phi(\ba\starM\bb)  &= \tube(M^{-1}\diag(M\vec(\ba))M\vec(\bb))\times_3 M\\
                        &= \tube(\diag(M\vec(\ba))M\vec(\bb))\\
                        &= \phi(\ba)\star_I\phi(\bb).
\end{align}
\end{subequations}
\end{proof}

The result of Proposition \ref{prop:All_iso} is that the commutative algebra alone is too coarse of a framework to effectively study the family of $\starM$-products. However, we notice that the map $\phi$ used in the proof of Proposition \ref{prop:All_iso} is simply a change of ($\Fbb$-vector space) basis on $\Fbb_M$. So, it is natural to ask “What is a good choice of basis?” We work towards an representation-theoretic answer to this question in Section \ref{sec:Equivariance}.

\subsubsection{Squares in the Ring $\bbR_M$} Nonnegative elements of the ring $\bbR$ are central to the theory of semidefinite matrices and squares in the ring $\bbR_M$ are the natural generalizations to the $\starM$-setting. 
Thus, we characterize elements of the ring $\bbR_M$ which are squares. 

\begin{lemma}\label{lem:squares}
Fix $M \in \GL_{n_3}(\Rbb)$. Then, $\ba \in \Rbb_M$ is a square if and only if $\vec(\widehat{\ba}) \in \bbR^{n_3}_+$ where $\widehat{\bfa} = \bfa \times_3 M$. 
\end{lemma}

\begin{proof} If $\ba \in \Rbb_M$ is a square, then $\vec(\ba) = M^{-1}\diag(Mx)Mx$ for some $x \in \bbR^{n_3}$.  Therefore, $\vec(\widehat{\ba}) = \diag(Mx)Mx \in \bbR^{n_3}_+$. 
Conversely,  if $\vec(\widehat{\ba}) \in \bbR^{n_3}_+$, then $\vec(\widehat{\ba}) = \diag(y) y$ for some $y \in \bbR^{n_3}$. 
 It follows that $\ba = \bfx \starM \bfx$ where $\bfx = \tube(M^{-1}y)$.  
\end{proof}

 \subsection{The $\starM$-Product} 
 
 As depicted in~\Cref{fig:tensor_notation}, tensors can be viewed as ``matrices of tubes.''  
Thus, the ring structure on the space of tubes suggests a natural tensor-tensor product for third-order tensors. 
Specifically,  we consider a tensor  $\cA\in \bbF^{n_1\times n_2 \times n_3}$ to be an $\bbF_M$-linear transformation of free modules $(\bbF_M)^{n_2} \to (\bbF_M)^{n_1}$. 
We thereby define the $\starM$-product of tensors analogously to matrix-matrix products.  
\begin{defn}[$\starM$-product]\label{def:mprod}
    Fix $M\in \GL_{n_3}(\Fbb)$. 
    Let $\cA \in \bbF_M^{n_1\times m}$ and $\cB \in \bbF_M^{m \times n_2}$. 
    Then, $\cC = \cA \starM \cB \in \bbF_M^{n_1\times n_2}$ where
    \begin{align}
    \cC_{i,j} = \sum_{\ell = 1}^{m}\cA_{i,\ell}\starM\cB_{\ell,j}.
\end{align}
for $i\in [n_1]$ and $j\in [n_2]$. 
\end{defn} 
Consequently, the $\starM$-product $\cA \starM \cB$ is the composition $(\bbF_M)^{n_2} \overset{\cB}{\to} (\bbF_M)^{m} \overset{\cA}{\to} (\bbF_M)^{n_1}$.  
 In practice, the $\starM$-product is implemented following \Cref{alg:Mprod}. 

\begin{algorithm}[t]
\caption{$\starM$-product}\label{alg:Mprod}
\KwData{$\cA \in \bbF^{n_1\times m \times n_3}, \cB \in \bbF^{m \times n_2 \times n_3}, M \in \GL_{n_3}(\bbF)$.}
\KwResult{$\cC = \cA \starM \cB \in \bbF^{n_1\times n_2\times n_3}$}
$\widehat{\cA} \gets \cA \times_3 M$, $\widehat{\cB} \gets \cB \times_3 M$\;
$\widehat{\cC} \gets \widehat{\cA} \smalltriangleup \widehat{\cB}$\;
$\cC \gets \widehat{\cC} \times_3 M^{-1}$.
\end{algorithm}

\Cref{alg:Mprod} shows that in the transform domain, the $\starM$ product operates independently on frontal slices. In terms of block diagonal matrices, $\bdiag(\hat{\cC}) = \bdiag(\hat{\cA})\bdiag(\hat{\cB})$. This provides one interpretation of the $\starM$ product in which the choice of matrix $M$ and passage to the transform domain amounts to a block diagonalization. 

An additional perspective on the $\starM$-product as a structured, block matrix operation will prove useful in subsequent theoretical analysis. 
We define the following tensor matricization operation. 
Let $\mymat_M: \Fbb_M^{n_1\times n_2} \to \Fbb^{n_1 n_3 \times n_2n_3}$ be defined as
	\begin{align}\label{eq:tensor_mat}
	\mymat_M(\cA) = \begin{bmatrix}
				\diag(M \vec(\cA_{1,1,:})) & \cdots & \diag(M \vec(\cA_{1,n_2,:}))\\
				\vdots & \ddots & \vdots\\
				\diag(M \vec(\cA_{n_1,1,:})) & \cdots & \diag(M \vec(\cA_{n_1,n_2,:}))
				\end{bmatrix}.
	\end{align}
The map $\mymat_M$ is an isomorphism onto its image, which consists of the block matrices with $n_1\times n_2$ elements of $D_{n_3}(\bbF)$. Using~\eqref{eq:tensor_mat}, we extend the matrix representative tubes in~\eqref{eq:matrix_representative} to an $\bbF$-linear transformation of the underlying $\bbF$-vector spaces $\bbF^{n_2 \times 1 \times n_3} \to \bbF^{n_1 \times 1 \times n_3}$. 
Specifically, given $M\in \GL_{n_3}(\Fbb)$,  let $T_{\cA}: \Fbb_M^{n_2} \to \Fbb_M^{n_1}$ be defined as $T_{\cA}(\cX) = \cA \starM \cX$ for any $\cX\in \Fbb_M^{n_2}$. 
Let $x = \mymat_I(\cX) \mathbbm{1}_{n_2n_3} \in \Fbb^{n_2n_3}$ where $I$ is the $n_3\times n_3$ identity matrix; effectively, $x$ is a specific vectorization of $\cX$. 
Then, 
	\begin{align}\label{eq:tensor_mat_full}
	T_{\cA}(\cX) \equiv (I_{n_1}\otimes M^{-1}) \mymat_M(\cA) (I_{n_2}\otimes M) x
	\end{align}
where $\otimes$ is the Kronecker product~\cite[Section 10.2]{petersen_matrix_2012}. 

\subsection{A Matrix Mimetic Algebra}

Because of the similarities between the $\starM$-product and traditional matrix-matrix multiplication, we say the $\starM$-product is \emph{matrix mimetic}. 
We can extend this similarity to other familiar linear algebraic properties. For example, the notions of multiplicative identity tensor and the transpose of a tensor make sense in the $\starM$ setting. 
The definitions follow from~\cite{kilmer_tensor-tensor_2021}, presented here in the language of rings.

\begin{defn}[$\star_M$-identity tensor]
The $\star_M$-identity tensor, denoted $\cI_M \in \bbF_M^{n\times n}$, has tubes 
	\begin{align}
	(\cI_M)_{i,j,:} = \begin{cases}\bfe_{M} & i = j,\\ 0 & \text{otherwise} \end{cases}
	\end{align}
for $i,j\in [n]$ where $\bfe_M = \tube(M^{-1} \mathbbm{1}_{n_3})$ is the $\starM$-multiplicative identity in $\Fbb_M$.  
The $\starM$-identity tensor satisfies $\cI_M \starM \cA = \cA \starM \cI_M = \cA$ for all $\cA \in \bbF_M^{n \times n}$. 
\end{defn}

\begin{defn}[$\starM$-conjugate transpose/$\starM$-transpose]\label{def:StarM_Transpose}

Fix $M\in \GL_{n_3}(\bbF)$. 
The $\starM$-conjugate transpose of $\cA \in \bbF_M^{n_1\times n_2}$, denoted $\cA^H \in \bbF_M^{n_2 \times n_1}$, is uniquely defined as 
    \begin{align}
        (\cA^H \times_3 M)_{:,:,k} = \left((\cA \times_3 M)_{:,:,k}\right)^H
    \end{align}
    for each $k \in [n_3]$.
If $\bbF = \bbR$, the notion of $\starM$-transpose, denoted $\cA^\top$, reduces to permuting the indices of tubes or equivalently transposing the frontal slices of $\cA$. 
\end{defn}

We note that the $\starM$-conjugate transpose satisfies matrix mimetic properties such as reversal of the order of the $\starM$-product under transposition; i.e., $(\cA \starM \cB)^H = \cB^H \starM \cA^H$.  
Further note that transposition commutes with the  tensor matricization operation in~\eqref{eq:tensor_mat}; that is, $\mymat_M(\cA^H) = (\mymat_M(\cA))^H$.

\begin{defn}[$\starM$-Hermitian/$\starM$-symmetric]\label{def:starm_symmetric}
Fix $M\in \GL_{n_3}(\bbF)$. A tensor $\cA \in \bbF_M^{n \times n}$ is $\starM$-Hermitian if $\cA^H = \cA$. 
If $\bbF = \bbR$ and $\cA^\top = \cA$, we call the tensor $\starM$-symmetric. 
\end{defn}

The definition of a $\starM$-symmetric tensor $\cA \in \bbR^{n\times n}_M$ provided in Definition \ref{def:starm_symmetric} is consistent with our point of view that third-order tensors are matrices of tubes. However, this differs from the notion of a symmetric order 3 tensor as an element of the symmetric power $\mathrm{Sym}^3(\bbR^n)$ encountered elsewhere. 

\begin{defn}[$\starM$-unitary/$\starM$-orthogonal]
For a fixed $M\in \GL_{n_3}(\bbF)$, a tensor $\cQ \in \bbF_M^{n\times n}$ is $\starM$-unitary if $\cQ^H \starM \cQ = \cQ \starM \cQ^H = \cI_M$. 
If $\bbF = \bbR$, we call a $\starM$-unitary tensor $\starM$-orthogonal. 
\end{defn}

A foundational result in the $\starM$-product literature is the existence of a tensor singular value decomposition (SVD) analogous to the matrix SVD. 

\begin{theorem}[Existence of the $\starM$-SVD {\cite{kernfeld_tensortensor_2015, kilmer_tensor-tensor_2021}}]\label{thm:SVD}
Fix $M\in \GL_{n_3}(\bbF)$ and let $\cA\in \Fbb_M^{n_1\times n_2}$. 
There exist $\starM$-unitary tensors $\cU\in \Fbb_M^{n_1\times n_1}$ and $\cV\in \bbF_M^{n_2\times n_2}$ and an f-diagonal tensor $\cS\in \bbF_M^{n_1\times n_2}$ such that 
\begin{equation}\tag{$\starM$-SVD}\label{eq:SVD}
\cA = \cU \starM \cS \starM \cV^H.
\end{equation}

\end{theorem}

The $\starM$-SVD leads to a notion of tensor rank. 

\begin{defn}[$\starM$-rank]\label{def:starm_rank}
Fix $M\in \GL_{n_3}(\bbF)$ and let $\cA\in \Fbb_M^{n_1\times n_2}$ with $\starM$-SVD given by $\cA= \cU \starM \cS \starM \cV^H$. 
The $\starM$-rank is the number of nonzero singular tubes in $\cS$; that is, 
	\begin{align}
		\starMrank(\cA) = \#\{\cS_{i,i,:} \mid \cS_{i,i,:} \not= {\bf0} \text{ for } i\in [\min(n_1,n_2)]\} .
	\end{align}
\end{defn}

Algorithmically, the $\starM$-SVD is formed by computing matrix SVDs facewise in the transform domain. Specifically, for each $k \in [n_3]$, we compute the matrix SVD
\begin{align}\label{eq:facewise_svd}
\widehat{\cA}_{:,:,k} = \widehat{\cU}_{:,:,k} \widehat{\cS}_{:,:,k} \widehat{V}_{:,:,k}^H
\end{align}
for each $k\in [n_3]$ where $\widehat{\cdot} = \cdot \times_3 M$. 
Note that the properties of the matrix SVD hold for each frontal slice in the transform domain; in particular, $\widehat{\cU}_{:,:,k}$ and $\widehat{\cV}_{:,:,k}$ are unitary and $\widehat{\cS}_{:,:,k}$ is diagonal with nonnegative, ordered diagonal entries $\widehat{\cS}_{1,1,k} \ge  \cdots \ge \widehat{\cS}_{r,r,k}  \ge 0$ where $r = \starMrank(\cA)$. 

\begin{remark}
If $\cA$ is $\starM$-symmetric, then we can take $\cU = \cV$ in the decomposition \eqref{eq:SVD}. This follows from the fact that the SVD of $\cA$ yields SVDs of the frontal slices of $\cA \times_3 M$. 
\end{remark}

\begin{remark}
For expositional clarity, we may use $\starM$- or $M$- when describing the properties of a tensor.  
We may omit the prefix when there is no ambiguity. 
\end{remark}

\subsection{Inner Product of Tensors Under the $\starM$-Product}

Central to constructing $\starM$-based tensor semidefinite programming is the inner product between tensors. 
We define the Frobenius-norm-induced inner product $\langle \cdot, \cdot \rangle : \Fbb^{n_1\times n_2\times n_3}\times \Fbb^{n_1\times n_2\times n_3} \to \Fbb$ as
	\begin{align}\label{eq:inner_prod}
	\langle \cA, \cB \rangle = \sum_{i = 1}^{n_1}\sum_{j = 1}^{n_2} \sum_{k = 1}^{n_3} \overline{\cA_{i,j,k}}\cB_{i,j,k}
	\end{align}
where $\overline{z}$ denotes the complex conjugate of a scalar $z\in \Fbb$. 
The Frobenius norm of a tensor induced by~\eqref{eq:inner_prod} is thereby $\|\cA\|_F = \sqrt{\langle \cA, \cA \rangle}$. 

Note that the definition of a tensor inner product is consistent with that of matrices (i.e., when $n_3 = 1$), which is alternatively presented as $\langle A, B\rangle = \mytrace(A^H B)$ for $A, B\in \Fbb^{n_1\times n_2}$. 
This observation gives rise to a facewise interpretation of~\eqref{eq:inner_prod}, given by
	\begin{align}\label{eq:inner_prod_facewise}
	\langle \cA, \cB \rangle = \sum_{k=1}^{n_3} \langle \cA_{:,:,k}, \cB_{:,:,k}\rangle.
	\end{align}
For completeness, we additionally provide a tubal perspective of~\eqref{eq:inner_prod} given by
	\begin{align}\label{eq:inner_prod_tube}
	\langle \cA, \cB \rangle = \sum_{i = 1}^{n_1}\sum_{j = 1}^{n_2} \langle \cA_{i,j,:}, \cB_{i,j,:}\rangle
	\end{align}
where $\langle \bfa, \bfb \rangle = \vec(\ba)^H \vec(\bb)$ for tubes $\bfa, \bfb\in \Fbb^{1\times 1\times n_3}$.

We next establish some useful identities about the inner product of real tensors and $\starM$ multiplication given by an orthogonal matrix $M$. These identities generalize familiar properties of the inner product of matrices. We note that the orthogonality of $M$ is a necessary hypothesis for the identities to hold. 

\begin{lemma}\label{lem:Trace_transpose}

Fix $M \in \Ogroup_{n_3}(\bbR)$ and let $\cX \in \bbR^{n \times 1 \times n_3}$.

\begin{enumerate}
    \item If $\cA,\cB \in \bbR^{1 \times n \times n_3}$, then

    \begin{equation}\langle \cA \starM \cX,\cB \starM \cX \rangle = \langle \cX, (\cA^\top \starM \cB) \starM \cX\rangle.\end{equation}

    \item If $\cA \in \bbR^{n \times n \times n_3}$, then

    \begin{equation}\langle \cX, \cA \starM \cX \rangle = \langle \cA, \cX \starM \cX^\top \rangle.\end{equation}

    \item For any tensors $\cA,\cB \in \bbR^{n_1 \times n_2 \times n_3}$, we have that 

    \begin{equation}\langle \cA \times_3 M, \cB \times_3 M \rangle = \langle \cA, \cB \rangle.\end{equation}
\end{enumerate}

\end{lemma}

\begin{proof} The proofs are calculations which we defer to \ref{sec:Appendix_Proof}.
\end{proof}

\section{Properties of $\starM$-semidefinite tensors}\label{sec:SemidefiniteTensors}

In this section, we provide a definition of $M$-positive semidefinite ($\starM$-PSD) third-order tensors which generalizes the definition of $t$-PSD tensors {\cite[Definition 3(ii)]{zheng_unconstrained_2022}}. 
We then prove basic properties of $\starM$-PSD tensors, showing that many properties from PSD matrices generalize to the $\starM$-product setting, and conclude with a discussion of $\starM$-semidefinite programming problems. 
Since our ultimate goal is the formulation of convex optimization problems, we restrict our attention to real-valued tensors with square frontal slices, i.e. $\cA \in \bbR_M^{n\times n}$ for $M\in \GL_{n_3}(\bbR)$. 
We start with an inner-product-based definition of $\starM$-PSD. 

\begin{tcolorbox}[colback=EmoryGold!5, colframe=EmoryGold]
\begin{defn}[$\starM$-positive semidefinite tensor]\label{def:MPSD}
Fix $M \in GL_{n_3}(\bbR)$.  
A symmetric tensor $\cA \in \bbR_M^{n \times n}$ is $\starM$-positive semidefinite ($\starM$-PSD) if for all $\cX \in \bbR_M^{n}$, we have $\langle \cX, \cA \starM \cX \rangle \geq 0$. 

We denote the set of all $\starM$-positive semidefinite tensors belonging to $\Rbb_M^{n\times n}$ by $\myPSD_M^n$. We write $\cA \succeq_M 0$ if $\cA \in \myPSD_M^n$.
\end{defn} 
\end{tcolorbox}

\subsubsection{Dependence on $M$} Note that Definition \ref{def:MPSD} depends on the choice of matrix $M$. For example, the $2\times 2\times 2$ tensor $\cX$ with frontal slices 
	\begin{align}
	\cX_{:,:,1} = \begin{bmatrix} 0 & 0 \\ 0 & 1\end{bmatrix} \quad \text{and } \quad \cX_{:,:,2} = \begin{bmatrix} 1 & 0\\ 0 &0\end{bmatrix}
	\end{align}
is $I$-PSD, but not $H$-PSD for $H = \frac{1}{\sqrt{2}}\begin{bmatrix} 1 & 1\\ 1 & -1 \end{bmatrix}$. This is discussed in more detail in Example \ref{ex:FeasibleRegions}. 

\subsubsection{Convexity of $\myPSD_M^n$}

From the $\bbR$-linearity of the $\starM$-product and the bilinearity of inner products, we see that if $\lambda,\mu \in \bbR_+$ are nonnegative constants and $\cA,\cB \in\myPSD_M^{n}$, then for any $\cX \in \bbR_M^{n}$,
	\begin{align}
	\langle \cX, (\lambda \cA + \mu \cB)\starM \cX \rangle = \lambda \langle \cX, \cA \starM \cX \rangle + \mu \langle \cX, \cB \starM \cX \rangle \geq 0.
	\end{align}
Therefore $\myPSD_M^n$ is a \emph{convex cone}, i.e. a convex set invariant under nonnegative scaling.

\subsection{Properties of $\starM$-PSD Tensors}
We derive basic properties of $\starM$-PSD tensors, emphasizing connections to the familiar theory of positive semidefinite matrices; see e.g., ~\cite[Appendix A.1]{blekherman_semidefinite_2012} for details. 
In~\Cref{prop:MPSD_BlockMat}, we show the matrix representative of an $M$-PSD tensor $\cA$ is a PSD matrix. 
In~\Cref{prop:slices}, we show that each frontal slice of $\cA \times_3 M$ is a PSD matrix. 
In~\Cref{prop:square_root}, we show the existence of a tensor factorization $\cA = \cB \starM \cB^\top$.  
In~\Cref{prop:M_det_square}, we show that the $\starM$-principal minors of $\cA$ are squares in $\Rbb_M$. 
We conclude with~\Cref{prop:square_eigentubes}, which contains a brief discussion of eigentubes. 
Similar properties were shown for the $t$-product in~\cite{zheng_t-positive_2021}, and our work generalizes these properties to any orthogonal $M$.

\subsubsection{Positive Semidefinite Matrix Representative} 
In a reasonable notion of $\starM$-positive semidefiniteness, we would expect the matrix representative of the underlying $\bbR$-linear transformation of a $\starM$-PSD tensor to be a PSD matrix itself. 
We prove this property below using the matrix representative in~\eqref{eq:tensor_mat_full}.

\begin{proposition}[PSD Matrix Representative]\label{prop:MPSD_BlockMat}
Fix $M\in \Ogroup_{n_3}(\bbR)$.
Given a symmetric tensor $\cA \in \bbR_M^{n\times n}$, we have that $\cA \in \myPSD_M^n$ if and only if $(I_n \otimes M^{\top}) \mymat_M(\cA) (I_n \otimes M)$ is positive semidefinite.
\end{proposition}

\begin{proof} For any $\cX\in \Rbb_M^n$, let $x \equiv  \mymat_I(\cX) \mathbbm{1}_{nn_3} \in \Rbb^{nn_3}$. 
Then, we have 
	\begin{subequations}
    \begin{align}
	\langle \cX, \cA \starM \cX \rangle& = \sum_{i = 1}^{n}\langle \cX_{i,1,:},(\cA \starM \cX)_{i,1,:}\rangle\\  &= \langle x,  (I_n \otimes M^{\top}) \mymat_M(\cA) (I_n \otimes M)x\rangle,\label{eq:PSDMatrixRepresentativeExpanded}
    \end{align}
    \end{subequations}
where we use the block structure of $x$ and $(I_n \otimes M^{\top}) \mymat_M(\cA) (I_n \otimes M)$ in \eqref{eq:PSDMatrixRepresentativeExpanded}. The equivalence of the tensor and matricized inner products combined with Definition~\ref{def:MPSD} and properties of PSD matrices in~\cite[Appendix A.1]{blekherman_semidefinite_2012} completes the bidirectional proof. 
\end{proof}

\subsubsection{Semindefiniteness in the Transform Domain} 
Most $\starM$-based algorithms are applied facewise in the transform domain, as this allows for independent matrix operations which can be parallelized.  
We show below that $\cA \succeq_M 0$ if and only if the frontal slices of $\hat{\cA}$ are PSD matrices. 

\begin{proposition}[PSD Transformed Frontal Slices]\label{prop:slices}
    Fix  $M \in \Ogroup_{n_3}(\Rbb)$. Let $\cA \in \bbR_M^{n \times n}$ be  symmetric. Then, $\cA$ is $\starM$-PSD if and only if $\bdiag(\widehat{\cA})$ is a positive semidefinite matrix. 
    
\end{proposition}

\begin{proof} Recall that a block diagonal matrix is PSD if and only if each block along the diagonal is PSD. The bidirectional proof proceeds as follows:

\begin{itemize}
\item[$(\Longrightarrow)$:]
If $\cA$ is $\starM$-PSD, then $\langle \cX, \cA \starM \cX \rangle \ge 0$ for all $\cX\in \Rbb_M^{n}$. 
Given $x \in \bbR^n$ and $k \in [n_3]$, construct $\cX^{(k)} = \widehat{\cX}^{(k)} \times_3 M^\top$ where 
	\begin{align}
		\widehat{\cX}_{:,:,\ell}^{(k)} = \begin{cases} x, & \ell = k\\ 0, & \ell \not=k \end{cases}.
	\end{align}

Then, by~\Cref{lem:Trace_transpose} and~\eqref{eq:inner_prod_facewise}, $\langle \cX^{(k)}, \cA \starM \cX^{(k)} \rangle = \langle x, \widehat{\cA}_{:,:,k} x\rangle \ge 0$ for all $k \in [n_3]$ and all $x\in \Rbb^n$. 
Thus, $\widehat{\cA}_{:,:,k} \succeq 0$ for all $k\in [n_3]$. 

\item[$(\Longleftarrow)$:] If each frontal slice of $\widehat{\cA}$ is PSD, then $\widehat{\cA}$ is $\star_I$-PSD by~\eqref{eq:inner_prod_facewise}. So, for $\cX \in \bbR^{n}_M$, we have by~\Cref{lem:Trace_transpose} that
\begin{align}
0 \leq \langle \cX, \hat{\cA}\star_I \cX \rangle = \langle \cX, (\cA\starM (\cX \times_3 M^\top))\times_3 M\rangle = \langle \cX \times_3 M^\top, \cA \starM (\cX \times_3 M^\top)\rangle. 
\end{align}

It follows that $\cA$ is $\starM$-PSD. 
\end{itemize}

\end{proof}

\subsubsection{Factorization of $\starM$-PSD tensors}\label{sec:Factor} 
Classical factorizations for PSD matrices, such as the Cholesky factorization, are central to mathematics. 
We provide a similar factorization here, starting with the case of $f$-diagonal tensors. 
Recall that a diagonal matrix $D \in D_n(\bbR)$ is positive semidefinite if and only if each element along the diagonal is nonnegative, or equivalently, each element along the diagonal is a square in $\bbR$. Since the ring of tubes $\bbR_M$ does not have an ordering, the analogous result for $f$-diagonal tensors asserts that the tubes along the diagonal are squares in the ring $\bbR_M$. 

\begin{lemma}\label{lem:diag_PSD}
    Fix an orthogonal matrix $M \in \Ogroup_{n_3}(\bbR)$. If $\cA \in \bbR_M^{n\times n}$ is an f-diagonal tensor with tubes $\cA_{i,i,:} = \ba_i$ then $\cA$ is $\starM$-PSD if and only if each $\ba_i$ is a square in the ring of tubes $\bbR_M$. 
\end{lemma}

\begin{proof} The bidirectional proof proceeds as follows.
\begin{itemize}
\item[$(\Longrightarrow)$:] Suppose that $\cA$ is $\starM$-PSD. Fix $i \in [n]$ and $\ell \in [n_3]$ and set $\cX \in \bbR^{n \times 1 \times n_3  }$ to have tubes 
\begin{align}
\bx_{j} = \begin{cases} \tube(M^\top e_\ell), & j = i\\ 0, & j \not = i
\end{cases}.
\end{align}
where $e_\ell$ is the $\ell^{\text{th}}$ standard basis vector of $\bbR^{n_3}$. Note that the tubes $\bx_j$ are defined so that the only nonzero entry of $\widehat{\cX}$ is $\widehat{\cX}_{i,1,\ell} = 1$. 
Then, 
\begin{align}
\langle \cX, \cA \starM \cX \rangle = (M^\top e_{\ell})^\top M^\top \diag(M\vec(\ba_i))e_\ell = e_\ell^\top \diag(M\vec(\ba_i)) e_{\ell}.
\end{align}
Since $\cA$ is $\starM$-PSD, it follows that $(M \vec(\ba_i))_{\ell} \geq 0$. Because $\ell$ was arbitrary, this implies that $\ba_i$ is a square in the ring of tubes by Lemma \ref{lem:squares}.

\item[$(\Longleftarrow)$:] If each $\ba_i = \bb_i \starM \bb_i$ for some $\bb_i \in \bbR_M$, then for any $\cX \in \bbR_M^n$, we have 
\begin{align}
\langle \cX, \cA \starM \cX \rangle = \sum_{i = 1}^{n} \langle \bx_i, \ba_i \starM \bx_i\rangle
= \sum_{i = 1}^{n} \langle \bb_i \starM \bx_i, \bb_i \starM \bx_i \rangle = \|\bb_i \starM \bx_i\|^2
\geq 0. 
\end{align}
\end{itemize}
\end{proof}

By passing to the $\starM$-SVD of $\cA$, we are able to conclude that all $\starM$-PSD tensors have the desired factorization. Additionally, passing to the $\starM$-SVD of $\cA$ connects the decomposition $\cA = \cB \starM \cB^\top$ to the $\starM$-rank of $\cA$ in a way that is analogous to PSD matrices. 

\begin{proposition}[$\starM$-PSD Square Root]\label{prop:square_root}
Fix $M \in \Ogroup_{n_3}(\bbR)$. Then, a $\starM$-symmetric tensor $\cA \in \bbR_M^{n \times n}$ is $\starM$-PSD if and only if $\cA = \cB \star_M \cB^\top$ for some tensor $\cB \in \bbR_M^{n \times r}$, where $r = \starMrank(\cA)$. 
\end{proposition}

\begin{proof} The bidirectional proof proceeds as follows.
\begin{itemize}
	\item[$(\Longrightarrow)$:] 
		From~\Cref{thm:SVD}, the $\starM$-SVD is given by $\cA = \cU \starM \cS \starM \cU^\top$. Then, for all $\cX \in \bbR^{n \times 1 \times n_3}$,

\begin{subequations}
\begin{align}
0 &\leq \langle (\cU\starM\cX) , \cA \starM (\cU\starM\cX) \rangle\\ 
&= \langle (\cU\starM\cX) , (\cU \starM \cS \starM \cU^\top) \starM (\cU\starM\cX) \rangle\\ 
&= \langle \cU^\top \starM (\cU\starM\cX), \cS \starM \cU^\top \starM (\cU\starM\cX) \rangle\\ &= \langle \cX, \cS \starM \cX \rangle.
\end{align}
\end{subequations}
		Following~\eqref{eq:facewise_svd} and~\Cref{lem:diag_PSD}, the nonzero tubes of $\cS$ are squares in $\Rbb_M$; that is, $\cS_{i,i,:} = \bfd_i \starM \bfd_i$ for some $\bfd_i\in \Rbb_M$ for $i\in [r]$.  
		Thus, we can write $\cS = \cD \starM \cD$ where $\cD \in \bbR_M^{n\times r}$ and $\cD_{i,i,:} = \bfd_i$ for $i\in [r]$. 
		Setting $\cB = \cU \starM \cD$ completes the proof. 
	\item[$(\Longleftarrow)$:] 
		From the matrix representative in~\eqref{eq:tensor_mat_full}, we have $\cB \starM \cB^\top \equiv (I_n \otimes M^\top) \mymat_M(\cB)\mymat_M(\cB)^\top (I \otimes M)$.  
		The matrix representative is PSD and thus $\cB\starM \cB^\top$ is $\starM$-PSD following \Cref{prop:MPSD_BlockMat}. 
\end{itemize}
\end{proof}

\subsubsection{Square principal minors of $\starM$-PSD Tensors}\label{sec:det} 
One important fact about PSD matrices is that all principal minors are nonnegative. This provides a system of polynomial inequalities which describes the PSD cone. 
In order to generalize this fact to tensors equipped with the $\starM$-product, we need a notion of $\starM$-determinant for tensors. 
Fix $M\in \Ogroup_{n_3}(\Rbb)$. 
For any $\cA \in \bbR_M^{n \times n}$, we define the $\starM$-determinant  as
\begin{align}\label{eq:M_det}
	\mydet_M(\cA) = \sum_{\sigma \in S_n} \sgn(\sigma)\cA_{1, \sigma(1),:}\starM \cA_{2,\sigma(2),:} \starM \cdots \starM \cA_{n,\sigma(n),:}
\end{align}
where $S_n$ is the order-$n$ symmetric group and $\sigma\in S_n$ is a permutation. 
The sign of a permutation $\sigma \in S_n$ is $\sgn(\sigma) = 1$ if $\sigma$ can be written as an even number of transpositions and $\mathrm{sgn}(\sigma) = -1$ otherwise\footnote{For example, in $S_3$,  the permutation $\sigma = (2 3)$ is odd (one swap of second and third entry; see e.g., \cite{dummit_abstract_2009} for cycle notation) with $\sgn(\sigma) = -1$. In comparison, $\sigma = (132) = (12)(13)$ is even (two swaps of first and third, then first and second entries) with $\sgn(\sigma) = 1$.}.  
Note that $\mydet_M(\cA) \in \bbR_M$ is a tube and that~\eqref{eq:M_det} is consistent with the point of view that $\cA$ is an $\bbR_M$-linear transformation of $\bbR_M$-modules. 
An equivalent facewise version of~\eqref{eq:M_det} is
	\begin{align}\label{eq:vec_detM}
	\mydet_M(\cA) \equiv M^{\top}\begin{bmatrix} \mydet(\widehat{\cA}_{:,:,1}) & \cdots & \mydet(\widehat{\cA}_{:,:,n_3})\end{bmatrix}^\top.
	\end{align}
This viewpoint is comparable to the $t$-product tubal determinant defined in \cite{el_hachimi_spectral_2023}.

We now have the tools to define a $\starM$-principal minor. 
Let $J = \{j_1,j_2,\ldots, j_\ell\} \subseteq [n]$ be a subset of indices with $1\le j_1 < j_2 < \cdots < j_\ell \le n$.  
Then, the $\starM$-principal minor of order $\ell$ indexed by $J$ is 
	\begin{align}\label{eq:principal_minor}
	\mydet_M(\cA_{J,J}) = \sum_{\sigma \in S_\ell} \mathrm{sgn}(\sigma)\cA_{j_1, j_{\sigma(1)},:}\starM \cA_{j_2, j_{\sigma(2)},:} \starM \ldots \starM \cA_{j_\ell, j_{\sigma(\ell)},:}.
	\end{align}

Using the definition \eqref{eq:principal_minor}, the following proposition generalizes the fact that PSD matrices have nonnegative principal minors to the $\starM$-PSD tensor setting. 

\begin{proposition}[Square $\starM$-Principal Minors]\label{prop:M_det_square}
Fix $M\in \Ogroup_{n_3}(\Rbb)$. 
A symmetric tensor $\cA \in \bbR_M^{n \times n}$ is $\starM$-PSD if and only if for each $J \subseteq [n]$, the principal minor $\mydet_M(\cA_{J,J})$ is a square in $\Rbb_M$. 
\end{proposition}

\begin{proof}
Let $J = \{j_1,j_2,\ldots, j_\ell\} \subseteq [n]$ with $1\le j_1 < j_2 < \cdots < j_\ell \le n$.  
By~\eqref{eq:principal_minor} and~\eqref{eq:vec_detM}, we have 
	\begin{align}
	\mydet_M(\cA_{J,J}) \equiv M^{\top}\begin{bmatrix} \mydet(\widehat{\cA}_{J,J,1}) & \mydet(\widehat{\cA}_{J,J,2}) & \cdots & \mydet(\widehat{\cA}_{J,J,n_3})\end{bmatrix}^\top.
	\end{align}
\begin{itemize}

\item[$(\Longrightarrow):$]By~\Cref{prop:slices}, each frontal slice $\widehat{\cA}_{:,:,k}$ is a PSD matrix and therefore each minor $\mydet(\hat{\cA}_{J,J,\ell})$ is nonnegative. 
It then follows from~\Cref{lem:squares} that $\mydet_M(\cA_{J,J})$ is a square in $\bbR_M$. 

\item[$(\Longleftarrow):$]Conversely, if  $\mydet_M(\cA_{J,J})$ is a square in $\bbR_M$ for any ordered set of indices $J$, then by~\Cref{lem:squares}, $\mydet(\widehat{\cA}_{J,J,k}) \geq 0$ for each $k \in [n_3]$.
Thus, every principal minor of the frontal slice $\widehat{\cA}_{:,:,\ell}$ is nonnegative  and therefore $\widehat{\cA}$ has PSD frontal slices. 
By~\Cref{prop:slices}, this implies $\cA \in \myPSD_M^n$.   
\end{itemize}
\end{proof}

\subsubsection{Eigentubes of $\starM$-PSD Tensors}\label{sec:eigentube} A final characterization of positive semidefinite matrices is that all eigenvalues are nonnegative. However, because the rings $\bbR_M$ have zero-divisors (nonzero elements whose product is zero), the $\starM$-PSD case is more subtle. We obtain a weaker statement regarding the relationship between $\starM$-PSD tensors and tubes which act analogously to eigenvalues. The spectral theory of tensors equipped with the $t$-product was developed in \cite{el_hachimi_spectral_2023}; see also \cite{gleich_power_2013}. Note also that a notion of spectral theory for higher order tensors which does not depend on a tensor-tensor product has been studied in the literature; see e.g. \cite{qi_tensor_2017} for an overview.

\begin{proposition}\label{prop:square_eigentubes}
    Fix $M\in \Ogroup_{n_3}(\Rbb)$. 
    Suppose that $\cA \in \myPSD_M^n$. Let $\bflambda\in \bbR_M$ and $\cX \in \bbR_M^{n}$ be such that $\cA \starM \cX = \bflambda \starM \cX$ and that all frontal slices of $\widehat{\cX}$ are nonzero. Then $\bflambda$ is a square in $\bbR_M$. 
\end{proposition}

\begin{proof}
If $\cA \starM \cX = \bflambda \starM \cX$, then in the transform domain, we have $\widehat{\cA} \smalltriangleup \widehat{\cX} = \widehat{\bflambda}\smalltriangleup\widehat{\cX}$, where $\smalltriangleup$ denotes the facewise product. It then follows that for each $k \in [n_3]$, 
	\begin{align}
	\widehat{\cA}_{:,:,k}\widehat{\cX}_{:,1,k} = \widehat{\bflambda}_{1,1,k}\widehat{\cX}_{:,1,k}.
	\end{align}
By the assumption that all frontal slices of $\widehat{\cX}$ are nonzero, this implies that $\widehat{\bflambda}_{1,1,k}$ is an eigenvalue of $\hat{\cA}_{:,:,k}$ for each $\ell \in [n_3]$. 
By~\Cref{prop:slices}, the matrix $\widehat{\cA}_{:,:,k}$ is positive semidefinite and therefore $\widehat{\bflambda}_{1,1,k}\geq 0$. It therefore follows from~\Cref{lem:squares} that $\bflambda$ is a square in $\bbR_M$. 
\end{proof}

Note that the requirement that all frontal slices of $\widehat{\cX}$ are nonzero is necessary for the conclusion of Proposition \ref{prop:square_eigentubes} to hold. For example, let $\cA \in \bbR^{2\times 2 \times 2}$ and $\cX \in \bbR^{2\times 1 \times 2}$ be given facewise by
	\begin{align}
	\cA_{:,:,1} = \cA_{:,:,2} = \begin{bmatrix} 1 & 0 \\ 0 & 1 \end{bmatrix}, \quad \cX_{:,1,1} = \begin{bmatrix} 1\\ 1\end{bmatrix}, \quad \text{and} \quad  \cX_{:,1,2} = \begin{bmatrix}0\\0 \end{bmatrix}. 
	\end{align}
Then, $\cA \in \myPSD_I^2$, where $I$ is the $2\times 2$ identity matrix, and $\cA \star_I \cX = \bflambda \star_I \cX$ for any $\bflambda\in \bbR_{I}$ with $\lambda_{1,1,1} = 1$. 
For example, $\bflambda = \tube(\begin{bmatrix} 1 & -1\end{bmatrix}^\top)$ satisfies the eigentube equation, but $\bflambda$ is not a square in $\Rbb_M$ because it contains a negative entry. 

\subsection{$M$-Semidefinite Programming}

Using the above description of $\starM$-PSD tensors, we define $M$-semidefinite programs to be the optimization of a linear functional over an affine slice of the cone $\myPSD_M^n$. 

\begin{tcolorbox}[colback=EmoryGold!5, colframe=EmoryGold]
\begin{defn}[$M$-SDP]
Let $M \in \Ogroup_{n_3}(\Rbb)$ and fix symmetric tensors $\cC \in \bbR_M^{n \times n}$ and $\cA^{(\ell)} \in \bbR_M^{n \times n}$  and scalars $b^{(\ell)}\in \Rbb$ for $\ell \in [L]$. 
An \emph{$M$-semidefinite programming problem} (M-SDP) is an optimization problem of the form
\begin{align}\label{eq:PMSDP}\tag{M-SDP}
\max_{\cX}\ \langle \cC, \cX \rangle \quad \text{s.t.} \quad \text{$\langle \cA^{(\ell)}, \cX \rangle = b^{(\ell)}$ for $\ell \in [L]$  and  $\cX \in \myPSD_M^n$}.
\end{align}
\end{defn}
\end{tcolorbox}

Problems of the form \eqref{eq:PMSDP} reduce to classical semidefinite programming in the case where $n_3 = 1$. 
On the other extreme, \eqref{eq:PMSDP} is equivalent to linear programming in $\bbR^{n_3}$ when $n = 1$ and $M = I$. In between these extreme cases, \eqref{eq:PMSDP} can be thought of as a standard SDP with $n_3$ matrix variables (or, equivalently, a single $nn_3 \times nn_3$ block diagonal matrix variable with blocks of size $n\times n$). 
Indeed, by~\Cref{prop:slices},  the condition $\cX \in \myPSD_M^n$ can be checked by the positive semidefiniteness of $\bdiag(\hat{\cX})$. 

In the special case where the affine constraints $\langle \cA^{(\ell)}, \cX \rangle = b^{(\ell)}$ are defined by tensors $\cA^{(\ell)}$ such that $\bdiag(\widehat{\cA^{(\ell)}})$ has only one nonzero $n\times n$ block (as illustrated in Figure \ref{fig:tsdp_constraint}), the problem \eqref{eq:PMSDP} can be solved as $n_3$ \emph{independent} matrix SDPs, resulting in reduced computational requirements. This fact is recorded below in Proposition \ref{prop:MSDP_block_variable} and heavily utilized in our numerical experiments on tubally sparse tensor completion in Section \ref{sec:Completion_Experiments}. 

\begin{figure}
    \centering
    \includegraphics[width=0.5\linewidth]{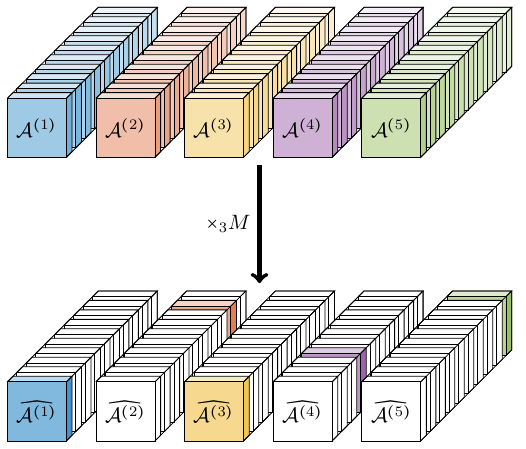}
    \caption{Illustration of linear constraint tensors for~\eqref{eq:PMSDP} that lead to $n_3$ independent matrix SDPs in the transform domain. 
    For illustration purposes, we consider $L = 5$ constraint tensors, each one depicted in a different color. 
    In the transform domain, for each $\ell\in [L]$, the tensor $\widehat{\Acal^{(\ell)}}$ has only one nonzero frontal slice, resulting in $\bdiag(\widehat{\Acal^{(\ell)}})$ having only one $n\times n$ nonzero block. 
    Note that the same frontal slices could be nonzero for different constraints, as in the case of $\widehat{\Acal^{(1)}}$ and $\widehat{\Acal^{(3)}}$.}
    \label{fig:tsdp_constraint}
\end{figure}

\begin{proposition}\label{prop:MSDP_block_variable}
Fix $M \in \Ogroup_{n_3}(\Rbb)$ and consider solving an $M$-semidefinite program of the form~\eqref{eq:PMSDP}. Suppose that all constraint tensors $\cA^{(\ell)}$ satisfy the property that $\widehat{\cA^{(\ell)}}$ has only one nonzero frontal slice, $\widehat{\cA^{(\ell)}}_{:,:,k_{\ell}}$ for some $k_{\ell} \in [n_3]$. Then, \eqref{eq:PMSDP} can be solved via $n_3$ independent matrix SDPs with decision variables of size $n\times n$. 

Explicitly, for each frontal slice $k \in [n_3]$, let $X^{(k)}$ be an optimal solution to the matrix SDP
\begin{align}
\gamma_{k} = \max_{X} \langle \widehat{C}_{:,:,k}, X \rangle \quad \text{s.t.} \quad \text{$\langle \widehat{\cA^{(\ell)}}_{:,:,k_{\ell}}, X \rangle = b^{(\ell)}$ for $\ell \in \{\ell \in [L] \; \vert \; k_{\ell} = k\}$  and  $X \succeq 0$}.
\end{align}
Then, the tensor $\cX$ such that $\widehat{\cX}_{:,:,k} = X^{(k)}$ for each $k \in [n_3]$ is an optimal solution to \eqref{eq:PMSDP} and the corresponding optimal objective value is $\gamma = \sum_{k = 1}^{n_3}\gamma_k$. 
\end{proposition}

\begin{proof}
By~\Cref{prop:slices}, we have $\cX \in \myPSD_M^n$ if and only if $\widehat{\cX}_{:,:,k} \succeq 0$ for each $k \in [n_3]$. Moreover, it follows from Lemma \ref{lem:Trace_transpose} and~\eqref{eq:inner_prod_facewise} that $\langle \cC, \cX \rangle = \langle \widehat{\cC}, \widehat{\cX} \rangle = \sum_{k = 1}^{n_3}\langle \widehat{\cC}_{:,:,k}, \widehat{\cX}_{:,:,k}\rangle$ and for each $\ell \in [L],$ we have $\langle \cA^{(\ell)}, \cX \rangle = \langle \widehat{\cA^{(\ell)}}, \widehat{\cX} \rangle = \sum_{k = 1}^{n_3}\langle \widehat{\cA^{(\ell)}}_{:,:,k}, \widehat{\cX}_{:,:,k}\rangle = \langle \widehat{\cA^{(\ell)}}_{:,:,k_\ell},\widehat{\cX}_{:,:,k_{\ell}}\rangle$.
\end{proof}

We conclude this section with an example highlighting the dependence on $M$ of the feasible region to \eqref{eq:PMSDP}. 

\begin{example}\label{ex:FeasibleRegions} 
Consider the $2\times 2 \times 2$ tensor $\cX$ with frontal slices
	\begin{align}	
		\cX_{:,:,1} = \begin{bmatrix} x & y\\ y & (1-x)\end{bmatrix} \qquad\text{and}\qquad \cX_{:,:,2} = \begin{bmatrix} (1-x) & y \\ y & x \end{bmatrix}.
	\end{align}
Now, $\cX \succeq_I 0$ if and only if $x(1-x)-y^2 \geq 0$. 
Thus, the set of $I$-PSD tensors with this structure is given by $(x,y)$ in the disk of radius $\frac{1}{2}$ centered at $(\frac{1}{2},0)$. 

On the other hand, if we set $\alpha = \frac{1}{\sqrt{2}}$ and $H = \alpha \begin{bmatrix} 1 & 1\\ 1 & -1 \end{bmatrix}$, then $\widehat{\cX} = \cX \times_3 H$ is 
	\begin{align}
	\widehat{\cX}_{:,:,1} = \begin{bmatrix}\alpha & 2\alpha y \\ 2\alpha y & \alpha\end{bmatrix} \qquad\text{and}\qquad \widehat{\cX}_{:,:,2} = \begin{bmatrix} \alpha(2x - 1) & 0 \\ 0 & \alpha(1-2x) \end{bmatrix}.
	\end{align}
Now, $\cX \succeq_{H} 0$ if and only if $x = \frac{1}{2}$ and $y^2 \leq \frac{1}{4}$. This results in a line segment in $\bbR^2$. 
The two regions are shown in Figure \ref{fig:FeasibleRegions}.

\begin{figure}
\centering

    \begin{subfigure}{0.4\linewidth}
    \centering
\begin{tikzpicture}
\pgfmathsetmacro{\X}{0}
\pgfmathsetmacro{\Y}{0}
\def\n{3}
\def\b{0.25}

\fill[white] (\X-2,\Y-2) rectangle (\X+\n,\Y+\n);
\draw (\X-2,\Y-2) grid (\X+\n,\Y+\n);
\draw[<->, line width=3pt] (\X-2,\Y) -- node[at end, right] {$x$} (\X+\n,\Y);
\draw[<->, line width=3pt] (\X,\Y-2) -- node[at end, above] {$y$} (\X,\Y+\n);

\node[fill=white, anchor=east] at (\X-0.25,\Y+1) {$\tfrac{1}{2}$};

\draw[blue, fill=blue!50, line width=2pt, fill opacity=0.75] (\X+1,\Y) circle (1);

\end{tikzpicture}

\subcaption{$\{(x,y) \mid \cX \succeq_I 0\}$ \vphantom{with $H = \tfrac{1}{\sqrt{2}} \begin{bmatrix} 1 & 1 \\ 1 & -1 \end{bmatrix}$}}
\end{subfigure}
\begin{subfigure}{0.4\linewidth}
    \centering
\begin{tikzpicture}
\pgfmathsetmacro{\X}{0}
\pgfmathsetmacro{\Y}{0}
\def\n{3}
\def\b{0.25}

\fill[white] (\X-2,\Y-2) rectangle (\X+\n,\Y+\n);
\draw (\X-2,\Y-2) grid (\X+\n,\Y+\n);
\draw[<->, line width=3pt] (\X-2,\Y) -- node[at end, right] {$x$} (\X+\n,\Y);
\draw[<->, line width=3pt] (\X,\Y-2) -- node[at end, above] {$y$} (\X,\Y+\n);

\node[fill=white, anchor=east] at (\X-0.25,\Y+1) {$\tfrac{1}{2}$};

\draw[magenta, line width=3pt] (\X+1,\Y+1) -- (\X+1,\Y-1);
\fill[magenta] (\X+1,\Y+1) circle (0.1);
\fill[magenta] (\X+1,\Y-1) circle (0.1);

\end{tikzpicture}
\subcaption{$\{(x,y) \mid \cX \succeq_H 0\}$ with $H = \tfrac{1}{\sqrt{2}} \begin{bmatrix} 1 & 1 \\ 1 & -1 \end{bmatrix}$.}
\end{subfigure}

\caption{The sets $\{(x,y) \; \vert \; \cX \succeq_I 0\}$ (left) and $\{(x,y) \; \vert \; \cX \succeq_H 0\}$ (right) discussed in~\Cref{ex:FeasibleRegions}. The region for which $\cX \succeq_I 0$ is a disk, while the region for which $\cX \succeq_H 0$ is a line segment.}\label{fig:FeasibleRegions} 
\end{figure}
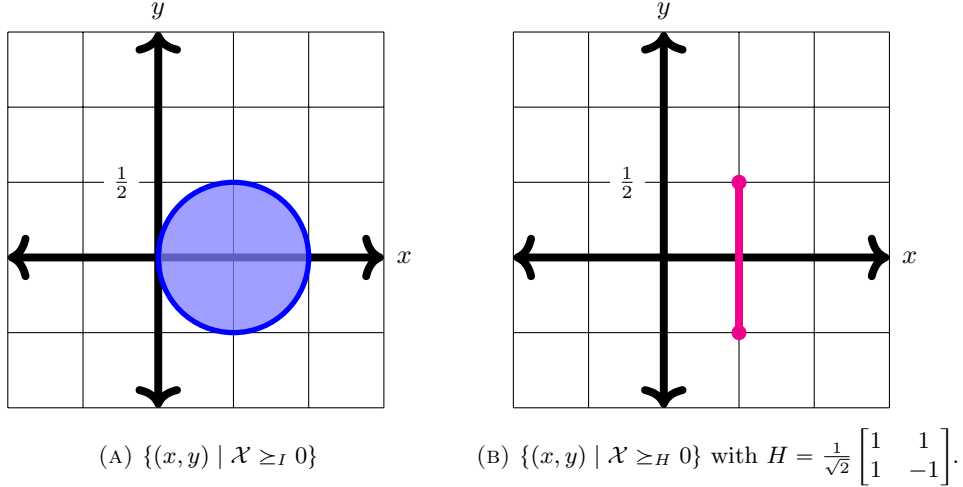

\end{example}

\section{Invariant SDPs and a Representation Theoretic Interpretation of the $\starM$-product}\label{sec:Equivariance}

A line of work in semidefinite programming, initiated by Gaterman and Parrilo \cite{gatermann_symmetry_2004}, studies semidefinite programs that are invariant under a group action. 
Such invariant presentations introduce new theoretical insights and numerical efficiencies. 
In particular, invariant SDPs can be naturally block-diagonalized, resulting in reduced computational cost. 

In this section, we connect \eqref{eq:PMSDP} to such invariant semidefinite programs. 
To do so, we develop an interpretation of the $\starM$-product based on the representation theory of finite groups. 
Specifically, we derive conditions on a finite group $G$ such that the $\starM$-product is equivariant under group actions; that is, for any $g\in G$ and under (left) group action denoted by $\cdot$, we have $\cA \starM (g \cdot \cB) = g \cdot (\cA \starM \cB)$ for any compatibly-sized tensors $\cA$ and $\cB$. 
These conditions are concrete and can be interpreted in terms of the solution to a linear system involving the transformation matrix $M$. 
The representation theoretic results may be of independent interest, so our presentation works in greater generality than is needed for $M$-SDP. 

\subsection{Brief Background on Group Representations} 
To preface the discussion of invariant $M$-SDPs, we briefly review representations of finite groups (see, e.g., standard references~\cite{serre_linear_1977, fulton_representation_2004}).

\subsubsection{Group Representations}\label{sec:RepTheoryBackground}
Let $G$ be a finite group. 
A \emph{(linear) representation of $G$} is a group homomorphism\footnote{Given two finite groups $G$ and $H$ with multiplication $\cdot_G$ and $\cdot_H$, respectively, we call $\varphi: G \to H$ a \emph{group homomorphism} if it preserves group structure; i.e., $\varphi(g \cdot_G g') = \varphi(g) \cdot_H \varphi(g')$.} $\rho:G \to \GL(V)$ for a finite-dimensional $\bbF$-vector space $V$.  
Here, $\GL(V)$ is the group of invertible linear transformations of $V$; i.e., $\GL(V) = \{\mu: V \to V \mid \text{$\mu$ invertible}\}$ and the group law is composition. We will only be concerned with the cases $\bbF = \bbR$ and $\bbF = \bbC$. 
A representation is \emph{faithful} if the map $\rho$ is injective.
In the representation theory literature, it is common to call the vector space $V$ a \emph{representation} of $G$ without explicit reference to the map\footnote{This is convenient when discussing properties independent of the choice of coordinates. However, in what follows, we will frequently emphasize choices of basis.} $\rho$. 

Let $W\subseteq V$ be a linear subspace of a representation. 
We say $W$ is \emph{$\rho$-invariant or $G$-linear} if $\rho(g) w\in W$ for all $g\in G$ and $w\in W$. 
A representation $V$ is \emph{irreducible} if there are no proper subspaces (i.e., $W \not= \emptyset$ nor $V$) which are $\rho$-invariant. 
Every representation can be written as the direct sum of irreducible representations.   
Specifically, we can decompose $V = V_1 \oplus V_2 \oplus \cdots \oplus V_m$ where $V_i \subseteq V$ is irreducible for $i\in [m]$.  
The direct sum notation means that in an appropriate basis, any $v\in V$ can be written uniquely in block form as $v = (v_1,v_2,\dots,v_m)$ for some $v_i \in V_i$ for $i\in [m]$. 
Note that  $\dim(V) = \dim(V_1) + \dim(V_2) + \cdots + \dim(V_m)$ and that we allow $V_i$ and $V_j$ to be isomorphic as representations. If we fix an ordered basis of an $\bbF$-vector space $V$ and $\dim(V) = n$, then $\GL_n(\bbF)$ is isomorphic to $\GL(V)$. 
Hence, we can attribute an invertible matrix $\rho(g)$ to the action of each group element $g\in G$. 
Consider the fixed basis $B$ of $V$ given by
\begin{equation} B = \{v_{1}^{(1)},v_{2}^{(1)},\ldots,v_{d_1}^{(1)},v_1^{(2)},\ldots v_{d_2}^{(2)},\ldots ,v_{d_m}^{(m)}\}\end{equation}
where $d_1 + d_2 + \cdots + d_m = \dim(V)$. 
We say $B$ of $V$ is \emph{compatible} with the decomposition of $V$ into irreducibles if  $\{v_1^{(i)},\dots v_{d_i}^{(i)}\}$ is a basis of $V_i$ for each $i \in [m]$. 
With a compatible basis, the matrices $\rho(g)$ will be block diagonal. 

Properties of representations are sensitive to the underlying field. A representation $\rho:G \to GL(V)$ for a complex vector space $V$ is called \emph{real} if $V  = V_0 \otimes \bbC$ for some real vector space $V_0$. Let $V_0$ be the underlying real vector space of a real representation and $V = V_0 \otimes \bbC$ its complexification. It is possible that $V_0$ has no nontrivial invariant $\bbR$-subspaces but that $V$ contains two irreducible (complex) subrepresentations. For more details, see \cite[Lecture 3.5]{fulton_representation_2004}. 
As an example, in general, if $G$ is an abelian group, then every irreducible representation over $\mathbb{F} = \bbC$ is one dimensional. However, this is not necessarily true over $\bbR$. Explicitly, the real representation of the cyclic group of order 4 with generator $\rho(g) = \begin{bmatrix} 0 & -1\\ 1 & 0\end{bmatrix}$ on the real vector space $\bbR^2$ is real irreducible, but the decomposition of $\bbC^2 = \bbR^2 \otimes \bbC$ into (complex) irreducibles is given by $\bbC^2 \cong \spann_\bbC \{(1,\sqrt{-1})\} \oplus \spann_\bbC \{(1, -\sqrt{-1})\}$. 

Let $V$ be a complex representation of a group $G$ and $V = V_1 \oplus V_2 \oplus \cdots \oplus V_m$ its decomposition into irreducibles. We say that $V$ is \emph{totally real} if each $V_i$ is real.

If $\rho:G \to GL_n(\bbF)$ is a representation and $T:\bbF^{n} \to \bbF^n$ is a linear transformation, we say that $T$ is \emph{$\rho$-equivariant} if $T(\rho(g)v) = \rho(g)T(v)$ for all $g \in G$ and $v \in \bbF^n$. Schur's Lemma characterizes the $\rho$-equivariant maps between irreducible representations of a group.  

\begin{lemma}[Schur's Lemma (see e.g.,{\cite{fulton_representation_2004}})]\label{lem:Schur}
Let $G$ be a finite group and let $\rho_V:G \to GL(V)$ and $\rho_W:G \to GL(W)$ be irreducible representations of $G$ over an algebraically-closed field $\Fbb$. Suppose that $T:V\to W$ is a linear map that satisfies $T(\rho_V(g)v) = \rho_W(g)T(v)$ for all $v \in V$ and $g \in G$.  
If $V = W$ and $\rho_V = \rho_W$, then $T = cI$ for some $c\in \bbF$ where $I$ is the identity transformation on $V$.
If $V \not\cong W$, then $T = 0$. 

\end{lemma}

Schur's Lemma gives places strong structure on equivariant linear transformations. Since we work primarily with matrix representatives of linear transformations, we need an appropriate choice of ordered basis to make this structure apparent.

\begin{defn}[Symmetry-Adapted Basis]\label{def:SymmetryAdaptedBasis}
Let $\rho:G \to \mathrm{GL}(V)$ be a linear representation of a finite group and $V = V_1 \oplus V_2\oplus \cdots \oplus V_m$ the decomposition of $V$ into irreducibles with $\dim V_i = d_i$ for $i\in [m]$.  An ordered basis of $V$ given by
\begin{equation}\label{eq:SymmetryAdaptedBasis}
B = \{ v_{1}^{(1)},v_{2}^{(1)},\ldots,v_{d_1}^{(1)}\}\cup \{v_1^{(2)}, v_{2}^{(2)},\ldots v_{d_2}^{(2)}\} \cup \cdots \cup \{v_1^{(m)}, v_{2}^{(m)}, \ldots,v_{d_m}^{(m)}\}
\end{equation}
is called a \emph{symmetry adapted basis} of $V$ if it satisfies the following two properties:
 
 \begin{enumerate}
 \item The set $\{v_1^{(i)},\dots v_{d_i}^{(i)}\}$ is a basis of $V_i$ for each $i \in [m]$.
 \item If $T:V \to V$ is a $\rho$-equivariant linear transformation, then its matrix representative with respect to $B$ has blocks $T_{i,j}:V_i \to V_j$ with $T_{i,j} = c_{i,j}I$ or $T_{i,j} = 0$.
 \end{enumerate}

\end{defn}

An important consequence of~\Cref{lem:Schur} and Definition~\ref{def:SymmetryAdaptedBasis} is that $\rho$-equivariant maps $T:V \to V$ can be represented using block diagonal structure. 
Indeed, by permuting a symmetry adapted basis, one obtains a block diagonal matrix. 

For example, if $V \simeq V_1 \oplus V_2$, where $V_1 \simeq V_2$ as representations, and $\dim(V_1) = 2$, an equivariant map $T:V \to V$ has the matrix representative 

\begin{equation}\begin{bmatrix}c_{11} I_2 & c_{12} I_2\\ c_{21}I_2 & c_{22}I_2 \end{bmatrix}\end{equation}
with respect to a symmetry adapted basis $B = \{v^{(1)}_1, v^{(1)}_2, v^{(2)}_1,v^{(2)}_2\}$ of $V$. Permuting this $B$ to the ordered basis $B' = \{v^{(1)}_1,v^{(2)}_1,v^{(1)}_2,v^{(2)}_2\}$ yields the matrix 

\begin{equation}\begin{bmatrix} c_{11} & c_{12} & &  \\ c_{21} & c_{22} & & \\ & & c_{11} & c_{12}\\ & & c_{21} & c_{22}\end{bmatrix} 
= \begin{bmatrix} 1 &  &  & \\  &  & 1 & \\  & 1 &  & \\  &  & & 1\end{bmatrix}\begin{bmatrix}c_{11} &  & c_{12} & \\ & c_{11} & & c_{12}\\ c_{21}&  & c_{22} & \\ & c_{21} & & c_{22} \end{bmatrix}\begin{bmatrix} 1 &  &  & \\  &  & 1 & \\  & 1 &  & \\  &  & & 1\end{bmatrix}.\end{equation}

\subsection{Equivariance of the $\starM$-Product}\label{sec:MProdEquivariance} 
Using the group representation background, we introduce group equivariance properties of the $\starM$-product.  
If $\rho:G \to \GL_{n_3}(\bbF)$ is a representation, then there is a left action of $G$ on $\bbF^{n_1\times n_2\times n_3}$ given by $g \cdot \cA = \cA \times_3 \rho(g)$ for any $\cA \in \bbF^{n_1\times n_2\times n_3}$ and $g \in G$. 
We first reduce the problem to determining the equivariance of the tubal multiplication map $T_{\ba}: \bbF^{1\times 1\times n_3} \to \bbF^{1\times 1 \times n_3}$ where, for a fixed $M\in \GL_{n_3}(\Fbb)$ and $\bfa\in \Fbb_M$, we have
	\begin{align}\label{eq:tubal_multiplication_map}
		T_{\bfa}(\bfx) = \bfa \starM \bfx
	\end{align}
for all $\bfx\in \Fbb^{1\times 1\times n_3}$. 
If tubal multiplication is $\rho$-equivariant, it follows from Definition~\ref{def:mprod} that the $\starM$-product of tensors is $\rho$-equivariant as well. 

To make the ideas of equivariance concrete, we take the $t$-product as our motivating example and show the corresponding multiplication map is equivariant under a representation of the cyclic group of order $n_3$.  

\begin{example}[Cyclic equivariance of the $t$-product]\label{ex:t_prod_equivariant}
Let $C_{n_3} = \langle g \mid g^{n_3} = e \rangle$ be the cyclic group of order $n_3$ with generator $g$ and identity element $e$. 
Let $\rho$ be a representation of $C_{n_3}$ where $\rho(g) = Z$, the circulant downshift matrix in~\eqref{eq:circulant_downshift}. 
We will show that the $t$-product is $\rho$-equivariant for this representation of $C_{n_3}$. 

Recall, the $t$-product is equivalent to the $\starM$-product using $M = F$, where $F$ is the discrete Fourier transform matrix in~\Cref{ex:t_prod}. Note also that the matrix $F$ gives a change of basis on $\bbC^{n_3}$ to a symmetry adapted basis. 

Following the tubal matrix representative of the $t$-product in~\eqref{eq:tprod_tubal_representation}, 
for any $\bfx\in \Cbb_F$, we have
    \begin{align}
        T_{\bfa}(g \cdot \bfx) = \bfa \star_F (g\cdot \bfx)
         &\equiv \mycirc(\bfa) \left(\rho(g) \myvec(\bfx)\right)
         =\left(\sum_{i=1}^{n_3}a_i \rho(g)^{i-1} \right) \left(\rho(g) \myvec(\bfx)\right).
    \end{align}
Because $\mycirc(\bfa)$ can be expressed as a linear combination of powers of $\rho(g)$, we can pass $\rho(g)$ through the summation; that is, 
    \begin{align}
        \left(\sum_{i=1}^{n_3}a_i \rho(g)^{i-1} \right) \left(\rho(g) \myvec(\bfx)\right) 
        = \rho(g)\left[\left(\sum_{i=1}^{n_3}a_i \rho(g)^{i-1}\right) \myvec(\bfx)\right]
        \equiv g \cdot T_{\bfa}(\bfx).
    \end{align}
\end{example}

\Cref{ex:t_prod_equivariant} shows that the map $T_\ba$ is $\rho$-equivariant for any $\ba\in \Cbb_F$. 
\Cref{thm:Equivariance} provides a characterization of the relationship between a group $G$, a representation $\rho$, and the matrix $M$ for this property to hold.

\begin{restatable}{theorem}{Equivariance}\label{thm:Equivariance}
Fix a finite group $G$, a representation $\rho:G \to GL_{n_3}(\bbF)$, and $M \in GL_{n_3}(\bbF)$.  
Then, the multiplication map $T_{\bfa}: \Fbb_M \to \Fbb_M$ is $\rho$-equivariant for all $\ba \in \bbF_M$ if and only if  $M\rho(g)M^{-1}$ is diagonal for each $g \in G$. 
\end{restatable}

\begin{proof}

The bidirectional proof proceeds as follows. 

\begin{itemize}
\item[$(\Longrightarrow):$]Suppose that $M\rho(g)M^{-1}$ is diagonal for each $g \in G$. Let $g \in G$ and $\ba,\bfb \in \Fbb_M$. 
Then, because diagonal matrices commute, we have
	\begin{subequations}
	\begin{align}
\ba \starM (g \cdot \bfb) &\equiv \left(M^{-1}\diag(M\vec(\ba))M\rho(g)\vec(\bb)\right)\\
&\equiv \left(M^{-1}\diag(M\vec(\ba))M\rho(g)M^{-1} M\vec(\bb)\right)\\
&=\left(M^{-1}M\rho(g)M^{-1}\diag(M\vec(\ba))M\vec(\bb)\right)\\
&\equiv g \cdot (\ba \starM \bb). 
	\end{align}
	\end{subequations}
	
\item[$(\Longleftarrow)$:] Conversely suppose that $T_{\bfa}$ is $\rho$-equivariant for all $\ba \in \bbF_M$. 
Then, for any $g \in G$ and $\bb \in \bbF_M$, we have $T_{\bfa}(g \cdot \bfb) = g\cdot T_{\bfa}(\bfb)$, or equivalently
	\begin{align}
	M^{-1}\diag(M\vec(\ba))M\rho(g)\vec(\bb) = \rho(g)M^{-1}\diag(M\vec(\ba))M\vec(\bb).
	\end{align}
Because this holds for any $\bfb\in \bbF_M$, we have
	\begin{align}
	\diag(M\vec(\ba))M\rho(g)M^{-1} = M\rho(g)M^{-1}\diag(M\vec(\ba)).
	\end{align}

Since $M$ is invertible, any diagonal $n_3\times n_3$ matrix can be written in the form $\diag(M\vec(\ba))$.  
In particular, for $\ell \in [n_3]$, let $E_\ell$ be the diagonal matrix with entry $e_{\ell\ell} = 1$ and all other entries zero. 
It follows that 
	\begin{align}
	E_\ell M\rho(g)M^{-1} = M\rho(g)M^{-1}E_\ell
	\end{align}
for all $\ell \in [n_3]$. 
To satisfy these $n_3$ equations,  $M\rho(g)M^{-1}$ must be diagonal.
\end{itemize}
\end{proof}

For many representations and groups of interest, the matrices $\rho(g)$ for $g \in G$ will not be simultaneously diagonalizable. Indeed, simultaneously diagonalizable matrices commute, but the group operation in many groups is noncommutative (e.g., the symmetric group $S_n$ for $n\geq 3$). 
Thus, the multiplication map $T_\ba$ will only be $\rho$-equivariant for some subset of tubes. 

\begin{corollary}\label{cor:Nonabelian_Nonequivariant}
Let $G$ be a nonabelian group and $\rho$ a faithful representation of $G$. For any $M\in \GL_{n_3}(\Fbb)$, there exist tubes $\ba \in \bbF_M$ for which the multiplication map $T_\ba$ is \underline{not} $\rho$-equivariant. 
\end{corollary}

\begin{proof}
If $\rho$ is a faithful representation of $G$ and $g,h \in G$ have $gh \not = hg$, then $\rho(g)\rho(h)\not = \rho(h)\rho(g)$. Because the matrices $\rho(g)$ and $\rho(h)$ do not commute, they are not simultaneously diagonalizable  
and the result follows by Theorem \ref{thm:Equivariance}. 
\end{proof}

Theorem \ref{thm:Equivariance} characterizes the representations $\rho$ and transformation matrices $M$ where \emph{all} multiplication maps $T_\ba$ are $\rho$-equivariant. In light of Corollary \ref{cor:Nonabelian_Nonequivariant}, it makes sense to shift perspectives and consider a \emph{fixed} representation $\rho$ and transformation matrix $M$ and characterize the subset $W_\rho \subseteq \bbF_M$ which satisfies the property that $T_\ba$ is $\rho$-equivariant for $\ba \in W_\rho$. We provide this characterization in Theorem \ref{thm:EquivariantSubspace} below. 

Recall from Section \ref{sec:RepTheoryBackground} that if $G$ is a finite group and $\rho$ is a representation, then the decomposition of $\bbC^{n_3}$ into irreducible representations of $\rho$ takes the form 

\begin{equation}\bbC^{n_3} \simeq V_1 \oplus V_2 \oplus \dots \oplus V_{m},\end{equation}
where $\dim V_j = d_j$. This gives rise to a\emph{symmetry-adapted basis} $B$ of $\bbC^{n_3}$.   

\begin{tcolorbox}[colback=EmoryBlue!5, colframe=EmoryBlue]
\begin{restatable}{theorem}{EquivariantSubspace}\label{thm:EquivariantSubspace}
Fix a finite group $G$ and a representation $\rho:G \to \GL_{n_3}(\bbC)$. Suppose that the decomposition of $\bbC^{n_3}$ into irreducibles is given by 
\begin{equation}\bbC^{n_3} \simeq V_1 \oplus V_2 \oplus \dots \oplus V_m,\end{equation}
where $\dim_{\bbC}(V_j) = d_j$ for $j\in [m]$. Let $M\in \GL_{n_3}(\bbC)$ represent a change of basis to a symmetry-adapted basis. Then, there is a vector subspace $W_\rho \subseteq \bbC_{M}$ for which $T_\ba$ is $\rho$-equivariant if and only if $\ba \in W_\rho$. $W_\rho$ is given explicitly by 
\begin{equation}\label{eq:EquivariantSubspace}W_\rho = \left\{ \bx\in \Cbb^{1\times 1\times n_3} \; \middle \vert \; \diag(M \vec(\bx)) = \begin{bmatrix} c_1I_{d_1} & & & \\ & c_2 I_{d_2} & & \\ & & \ddots & \\ & & & c_m I_{d_m} \end{bmatrix}, \enskip c_1,c_2,\ldots, c_m \in \bbC \right\}. \end{equation}

\end{restatable}
\end{tcolorbox}

\begin{proof}

The map $T_\ba$ is $\rho$-equivariant if and only if 
\begin{equation}
M^{-1}\diag(M\vec(\ba))M \rho(g) = \rho(g)M^{-1}\diag(M\vec(\ba))M,\end{equation}
which is equivalent to $\diag(M\vec(\ba))$ and $M\rho(g)M^{-1}$ commuting. As a result, if $\hat{\rho}:G \to \GL_{n_3}(\bbC)$ is the representation with $\hat{\rho}(g) = M\rho(g)M^{-1}$, then the linear transformation given by multiplication by $\diag(M\vec(\ba))$ is $\hat{\rho}$-equivariant. Note that since $M$ is a change of basis to a symmetry-adapted basis, the decomposition of $\bbC^{n_3}$ into $\hat{\rho}$-irreducibles is given by subspaces spanned by standard basis vectors:
\begin{equation}
\bbC^{n_3} \simeq \spann\{e_1,\ldots, e_{d_1}\} \oplus \spann \{e_{d_1 + 1}, \ldots e_{d_1 + d_2}\} \oplus \dots \oplus \spann \left\{e_i \; \middle\vert \; \left(\sum_{j = 1}^{m-1}d_j\right) + 1 \leq i \leq n_3\right\}.
\end{equation}

By Schur's Lemma (Lemma \ref{lem:Schur}), multiplication by $\diag(M\vec(\ba))$ map is $\hat{\rho}$-equivariant if and only if the restrictions to each $\hat{\rho}$-irreducible are given by multiplication by a constant. This happens if and only if $\diag(M\vec(\ba))$ has the desired block structure in \eqref{eq:EquivariantSubspace}.  
\end{proof}

Theorem \ref{thm:EquivariantSubspace} gives explicit conditions linear conditions on the space of tubes for $T_\ba$ to be equivariant in terms of the rows of the matrix $M$. To utilize these conditions, we form an auxillary matrix $V \in \bbC^{n_3\times m}$, with block structure

\begin{equation}V = 
\begingroup 
\setlength\arraycolsep{0pt}
\begin{bmatrix} \mathbbm{1}_{d_1} \\  & \mathbbm{1}_{d_2} \\ &  & \rotatebox{-60}{$\cdots$} & \\ & &  & \mathbbm{1}_{d_m}\end{bmatrix},
\endgroup
\end{equation}
where $\mathbbm{1}_{d_i} \in \Cbb^{d_i\times 1}$ is the constant vector of all ones.

Then, for any $a\in \Cbb^{n_3}$, $\tube(a) \in W_\rho$ if and only if we can find coefficients $c\in \Cbb^m$ such that $Ma = Vc$. 

We demonstrate the statement of Theorem \ref{thm:EquivariantSubspace} on the permutation representation of the symmetric group $S_3$. 

\begin{example}[The Symmetric Group $S_3$]\label{ex:symmetric_group}

Consider the symmetric group 
    \begin{align}
        S_3 = \langle \sigma,\tau \mid \sigma^2 = \tau^3 = e, \sigma \tau = \tau^2 \sigma\rangle.
    \end{align}
Let $\rho:S_3 \to GL_3(\bbC)$ be the permutation representation so that  

\begin{equation}\rho(\sigma) = \begin{bmatrix} 0 & 1 & 0\\ 1 & 0 & 0 \\ 0 & 0 & 1 \end{bmatrix} \quad \text{and} \quad \rho(\tau) = \begin{bmatrix} 0 & 0 & 1\\ 1 & 0 & 0\\ 0 & 1 & 0 \end{bmatrix}.\end{equation}

The decomposition of $\bbC^3$ into irreducibles is then $\bbC^3 \simeq V_1 \oplus V_2$ for $V_1 =\spann_{\bbC}\{(1,1,1)\} $ and $V_2 =\spann_{\bbC}\{(1,-1,0), (1, 0,-1)\} $. 
Let the bases of the irreducibles form the columns of $M$; that is, 

\begin{align}
    M = \begin{bmatrix} 1 & 1 & 1\\1 & -1 & 0\\ 1 & 0 & -1\end{bmatrix}
    \quad \text{and} \quad 
    M^{-1} = \frac{1}{3}\begin{bmatrix}
    1 & 1 & 1\\
    1 & -2 & 1\\
    1 & 1 & -2
    \end{bmatrix}.
\end{align}

To determine the $\rho$-equivariant transformations, $T_{\ba}$, we solve for $\bfa$ using the relation $Ma = Vc$. 
In particular, we compute the kernel of the augmented matrix $\left[\begin{array}{c|c} M & V \end{array}\right]$; that is, 
    \begin{align}
    \ker\left(
        \left[\begin{array}{ccc|cc}
        1 & 1 & 1 & 1 & 0\\
        1 & -1 & 0 & 0 & 1\\ 
        1 & 0 & -1 & 0 & 1 
        \end{array}\right]
        \right)
         = \left\{ a\in \bbC^3, c_1,c_2\in \bbC \; \left\vert\; \begin{array}{rcl} a_1 + a_2 + a_3 &=& c_1\\  a_1 - a_2 &=& c_2 \\ a_1 - a_3 &=& c_2 \end{array}\right.\right\}.
    \end{align}
Theorem \ref{thm:EquivariantSubspace} then implies that $T_\ba$ is $\rho$-equivariant if and only if $a_1-a_2 = a_1 - a_3 = c_2$ for some constant $c_2\in \bbC$.
Therefore, the set of tubes which yield $\rho$-equivariant transformations are
    \begin{align*}
        W_{\rho} = \{\ba \in \bbC^{1\times 1 \times 3}\mid a_1-a_2 = a_1 - a_3\} = \spann_{\bbC}\{\tube((1,0,0)), \tube((0,1,1))\}
    \end{align*}
Let $\bfa_1 \equiv (1,0,0)$ and $\bfa_2 \equiv (0,1,1)$ be the elements of a basis tubes of $W_{\rho}$. 
Let $\bfb\in \Cbb^{1\times 1\times 3}$ be arbitrary and consider $T_{\bfa}(\bfb)$ for each basis vector; that is, 
\begin{subequations}
    \begin{alignat}{4}
        T_{\bfa_1}(\bfb) &= \bfa_1 \starM \bfb &&\equiv \underbrace{M^{-1} \diag(M \vec(\bfa_1)) M}_{I_3} \vec(\bfb) &&= \begin{bmatrix}
            b_1 \\ b_2 \\ b_3
        \end{bmatrix}, \quad \text{and} \label{eq:Ta1}\\
        T_{\bfa_2}(\bfb) &= \bfa_2 \starM \bfb &&\equiv \underbrace{M^{-1} \diag(M \vec(\bfa_2)) M}_{Z_3} \vec(\bfb) &&= \begin{bmatrix}
            b_2 + b_3 \\ b_1 + b_3 \\ b_1 + b_2
        \end{bmatrix} \label{eq:Ta2}.
    \end{alignat}
\end{subequations}
Here, $Z_3 = \mathbbm{1}_3 (\mathbbm{1}_3)^\top - I_3$ where $\mathbbm{1}_3$ is the $3\times 1$ constant vector of all ones. 
Both transformations are $\rho$-equivariant; that is, for any $g\in S_3$, $T_{\bfa_i}(g \cdot \bfb) = g \cdot T_{\bfa_i}(\bfb)$ for $i=1,2$. 
This result is immediate for $T_{\bfa_1}$, which is equivalent to the identity transformation in~\eqref{eq:Ta1}. 
For $T_{\bfa_2}$, because $Z_3$ commutes with permutation matrices, we have 
    \begin{align}
        T_{\bfa_2}(g \cdot \bfb)
        \equiv Z_3 \rho(g) \vec(\bfb)
        = \rho(g) Z_3 \vec(\bfb) \equiv g \cdot T_{\bfa_2}(\bfb). 
    \end{align}
Thus, $T_{\bfa}$ is $\rho$-equivariant for any $\bfa\in W_{\rho}$. 
We illustrate the geometry of the $\rho$-equivariance for $S_3$ in~\Cref{fig:symmetric_intuition}. 

\end{example}

\begin{figure}
\centering
\includegraphics[width=0.6\linewidth]{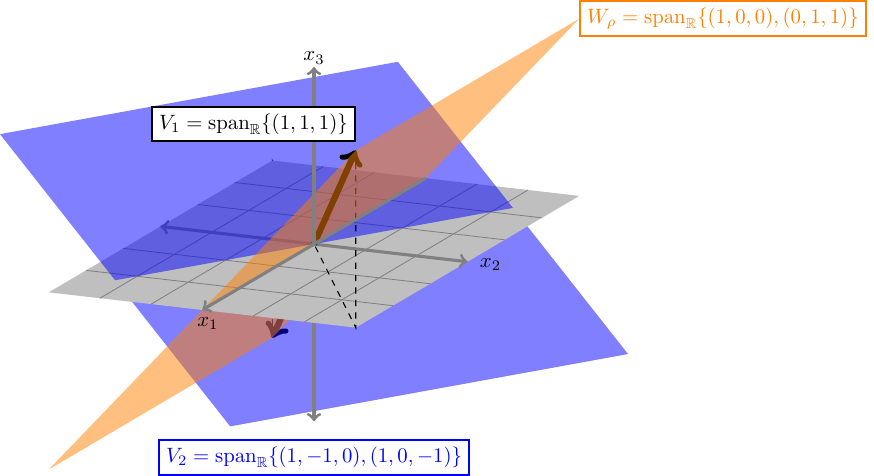}

\caption{Geometric interpretation of $\rho$-equivariance of symmetric group $S_3$ under the $\starM$-product from~\Cref{ex:symmetric_group}. 
The thick {\bf black} line through the origin depicts $V_1 = \spann_{\Rbb}\{(1,1,1)\}$ and the {\bf \color{blue}blue} plane depicts $V_2 = \spann_{\Rbb}\{(1,-1,0),(1,0,-1)\}$ as the irreducibles of $\Rbb^3$. 
The {\bf \color{orange} orange} plane shows the $\rho$-equivariant set, $\spann_{\Rbb}\{(1,0,0),(0,1,1)\}$.} 
\label{fig:symmetric_intuition}
\end{figure}

\subsection{Connection to Invariant SDPs} We now use the representation theoretic interpretation of the $\starM$-product to study invariant semidefinite programs. The equality constraints in \eqref{eq:PMSDP} allow one to enforce that multiplication by the tensor $\cX$ satisfies group equivariance properties. Indeed, if $\rho:G \to GL_{n_3}(\bbR)$ and $M$ are as in Theorem \ref{thm:EquivariantSubspace}, then the condition that multiplication is $\rho$-equivariant is a linear condition on the space of tubes and can therefore be written using equalities of the form $\langle \cA^{(i)}, \cX \rangle = 0$ for some symmetric tensors $\cA^{(i)}$. In particular, group equivariance can be encoded in the constraints of \eqref{eq:PMSDP}, and such cases yield invariant semidefinite programs.

\subsubsection{Invariant Matrix SDPs} 
Recall, a matrix SDP is the optimization of a linear functional over an affine slice of the cone of positive semidefinite matrices, which lies in the vector space 
$\Sbb^n = \{Y\in \Rbb^{n\times n} \mid X^\top = X\}$ of all $n\times n$ real-valued symmetric matrices. 
A representation $\rho: G \to GL_{n}(\bbR)$ gives $\Sbb^n$ the structure of a representation via the action 

	\begin{align}\label{eq:action_symmetric}
		g \cdot X = \rho(g)^\top X \rho(g).
	\end{align}
A matrix SDP is called \emph{invariant} if the objective function remains constant and the feasible region is fixed under group actions. That is, if $\Omega \subseteq \bbS^n$ is the feasible region and $C$ is the cost matrix, 
then $\Omega = g\cdot \Omega := \{g \cdot X\; \vert \; X\in \Omega\}$ for each $g\in G$ and $\langle C,g\cdot X \rangle = \langle C, X \rangle$ for each $g\in G$ and $X\in \Omega$. 

A central result of \cite{gatermann_symmetry_2004} is that an SDP which is invariant to a totally real representation of a group $G$ can be block-diagonalized by changing basis. 
Specifically, if $\rho(g)$ is orthogonal for every $g\in G$, then a symmetric matrix $X$ is fixed under the action of $G$ if and only if $\rho(g) X = X \rho(g)$. 
In this setting, the linear transformation $T_X: \bbR^{n_3} \to \bbR^{n_3}$ with $T_X(v) = Xv$ is $\rho$-equivariant. 
By permuting a symmetry adapted basis as in Definition \ref{def:SymmetryAdaptedBasis}, one obtains a basis for which $X$ is  block diagonalized. 

\subsubsection{$M$-SDPs and Invariant Matrix SDPs}

The results of Section \ref{sec:SemidefiniteTensors} imply that $M$-SDP can be viewed as a block-diagonalization of large semidefinite programs. Section \ref{sec:MProdEquivariance} provides an interpretation of the $\starM$-product in terms of representation theory. We connect these two ideas using the framework of invariant SDPs to characterize the case where these points of view are comptaible. That is, the case where the block diagonalization arising from the $\starM$-product structure agrees with the structure of an invariant SDP. 

\begin{theorem}\label{thm:InvariantSDP}
Let $G$ be a finite group and $\rho:G\to GL_{n_3}(\bbR)$ a totally real representation with $\rho(g)$ orthogonal for each $g \in G$. Let $M\in\Ogroup_{n_3}(\Rbb)$ and $W_\rho$ be as in the statement of Theorem \ref{thm:EquivariantSubspace}. Set $U \subseteq W_{\rho}$ to be the vector space of symmetric tensors whose tubes are elements of $W_\rho$. Then, for any symmetric tensors $\cC,\cA^{(\ell)}$ of size $n\times n \times n_3$ and any scalars $b^{(\ell)}$ for $\ell \in [L]$, the $M$-SDP
\begin{align}
\max \enskip \langle \cC, \cX \rangle \enskip \text{s.t. } \langle \cA^{(\ell)},\cX \rangle = b^{(\ell)} \text{ for } \ell \in [L], \enskip \cX \in \mathrm{PSD}_M^n \cap U
\end{align}
is an invariant $M$-SDP with respect to the representation $\hat{\rho}:G \to GL_{nn_3}(\bbR)$ given by $\hat{\rho}(g) = I_n \otimes \rho(g)$, where $I_n$ is the $n\times n$ identity matrix and $\otimes$ is the matrix Kronecker product. 
\end{theorem}

\begin{proof}
Recall from \eqref{eq:tensor_mat_full} that for $\cX \in \bbR_M^{n \times n}$, the multiplication map $T_\cX: \bbR^{n\times 1 \times n_3} \to \bbR^{n \times 1 \times n_3}$ gives a linear transformation $\bbR^{nn_3} \to \bbR^{nn_3}$ with matrix representative
\begin{align}\hat{X} &= (I_n \otimes M^{\top}) \mymat_M(\cX)(I_n \otimes M).
\end{align}
Note that $\hat{X}$ is a $nn_3\times nn_3$ block matrix with blocks $M^\top D_{i,j} M$, where $D_{i,j} = \diag(M\vec(\bx_{i,j})) \in \Dgroup_{n_3}(\bbR)$.
Now, if $\cX \in U$, then $\bx_{i,j} \in W_\rho$ for each $i,j$ and therefore $M^{-1}D_{i,j}M$ commutes with $\rho(g)$ for any $g \in G$. Since $\rho(g)^\top = (\rho(g))^{-1} = \rho(g^{-1})$, we have
\begin{equation}\rho(g)^{\top}M^{\top}D_{i,j}M = M^{\top}D_{i,j}M\rho(g)^\top.\end{equation}
It then follows that for any $g\in G$, 
\begin{equation}\hat{\rho}(g)^\top \hat{X}\hat{\rho}(g) = \hat{X}\hat{\rho}(g)^\top\hat{\rho}(g) = \hat{X}. \end{equation}

To conclude, we need to show that if $\cB \in \bbR^{n \times n \times n_3}$ is a symmetric tensor and $\hat{B}$ is its matrix representative, then $\langle \cB,\cX \rangle = \tr(\hat{B}\hat{X})$. Note that if $\ba,\bb \in \bbR^{1\times 1 \times n_3}$ are tubes, then, because $M$ is orthogonal, 
\begin{align}
\begin{split}
\langle \ba,\bb\rangle&= \langle M\vec(\ba), M\vec(\bb)\rangle\\
&= \tr(\diag(M\vec(\ba))\diag(M\vec(\bb)))\\
&= \tr(M^{\top}\diag(M\vec(\ba))MM^{\top}\diag(M\vec(\bb))M). 
\end{split}
\end{align}

It then follows that 
\begin{align}
\begin{split}
\langle \cB,\cX \rangle
&= \sum_{i = 1}^n \left(\sum_{k = 1}^n\langle \bb_{i,k}, \bx_{k,i}\rangle\right)\\
 &= \sum_{i=1}^n\left(\sum_{k = 1}^n\tr(M^{\top}\diag(M\vec(\bb_{i,k}))MM^{\top}\diag(M\vec(\bx_{k,i}))M)\right)\\
&= \tr(\hat{B}\hat{X}).
\end{split}
\end{align}
\end{proof}

We conclude this section with a simple example to highlight the connection between $M$-SDPs and invariant matrix SDPs.

\begin{example}\label{ex:Invariant_SDP} 
Consider the symmetric group $S_2 = \langle e, \sigma \mid \sigma^2 = e \rangle$. 
Let $\rho:S_2 \to GL_2(\bbR)$ be the representation with 
    \begin{align}
        \rho(\sigma) = \begin{bmatrix} 0 & 1\\ 1 & 0 \end{bmatrix}. 
    \end{align}
The decomposition of $\Rbb^2$ into irreducible representations is $\Rbb^2 \simeq \spann_\bbR\{(1,1)\} \oplus \spann_\bbR\{(1,-1)\}$.
Note that each irreducible has multiplicity 1. 

Consider the orthogonal matrix $H = \alpha \begin{bmatrix} 1 & 1 \\ 1 & -1\end{bmatrix}$, where $\alpha = \frac{1}{\sqrt{2}}$. 
Since $H\rho(\sigma)H^\top = \begin{bmatrix}1 & 0 \\ 0 & -1 \end{bmatrix}$ is diagonal, it follows from Theorem \ref{thm:Equivariance} that $T_\ba$ is $\rho$-equivariant for all tubes $\ba \in \bbR^{1\times 1 \times 2}$. That is, with notation as in the statement of Theorem \ref{thm:InvariantSDP}, we have $U = \bbR^{2\times 2 \times 2}$ and therefore $\mathrm{PSD}^2_H \cap U = \mathrm{PSD}^2_H$.

Consider the following tensor $\cX$ of format $2\times 2\times 2$ with variable entries $x_1,x_2,y_1,y_2$ and its image in the transform domain $\widehat{\cX} = \cX \times_3 H$ given by 
    \begin{subequations}
    \begin{align}
        \cX_{:,:,1} &= \begin{bmatrix}
            x_1 & 1 \\ 1 & y_1
        \end{bmatrix}, 
         &    \cX_{:,:,2} &= \begin{bmatrix}
            x_2 & 1 \\ 1 & y_2
        \end{bmatrix}\\
    \widehat{\cX}_{:,:,1} &= \begin{bmatrix}
            \alpha(x_1 + x_2) & 2\alpha \\ 2\alpha & \alpha(y_1 + y_2)
        \end{bmatrix},  \quad \text{and}
         &     
         \widehat{\cX}_{:,:,2} &= \begin{bmatrix}
            \alpha(x_1 - x_2) & 0 \\ 0 & \alpha(y_1 - y_2)
        \end{bmatrix}.
    \end{align}
    \end{subequations}
We will show that the corresponding $M$-SDP
\begin{equation}\label{eq:Example_MSDP_Invariant}
\begin{aligned}
\min_{(x_1,x_2),(y_1,y_2)} \, & \langle \cI, \cX \rangle
\qquad \text{s.t. } \qquad \cX \succeq_H 0
\end{aligned}
\end{equation}
is invariant under the group action of $S_2$, where $\Ical$ is the facewise identity tensor defined in \Cref{sec:notation}.

By Proposition \ref{prop:MSDP_block_variable}, we can rewrite \eqref{eq:Example_MSDP_Invariant} using block matrices with blocks from the transform domain resulting in the equivalent $M$-SDP
\begin{equation}\label{eq:Example_MSDP_Invariant_Transform}
\begin{aligned}
\min_{(x_1,x_2),(y_1,y_2)} \, & \langle \widehat{\cI}, \widehat{\cX} \rangle  \qquad 
\text{s.t. } \, \qquad \begin{bmatrix}\widehat{\cX}_{:,:,1}\\
& \widehat{\cX}_{:,:,2}\end{bmatrix}\succeq 0.
\end{aligned}
\end{equation}
Because the frontal slices in the transform domain are symmetric matrices, the block diagonal matrix $\bdiag(\widehat{\Xcal})$ is a symmetric matrix that is invariant under the induced action $\hat{\rho}(\sigma) = I_2 \otimes \rho(\sigma)$. Moreover, the objective function $\langle \widehat{\cI}, \widehat{\cX} \rangle = \langle \cI,\cX \rangle = x_1 + x_2 + y_1 + y_2$ is invariant under $\hat{\rho}$, which acts as permutations $(x_i,y_i) \mapsto (x_{g(i)},y_{g(i)})$ for $g \in S_2$. Since the objective function and feasible region of \eqref{eq:Example_MSDP_Invariant_Transform} are both invariant under $\hat{\rho}$, the problem \eqref{eq:Example_MSDP_Invariant_Transform} is an invariant SDP.  

Note that the decomposition of $\bbR^4$ into irreducible representations of $S_2$ consists of the trivial representation and the sign representation, each with multiplicity 2, and the corresponding blocks in \eqref{eq:Example_MSDP_Invariant_Transform} are $2\times 2$.  

\end{example}

\section{Applications}\label{sec:Applications}

In this section, we describe two applications of the $M$-SDP framework. Our first application is a description of certain group invariant sums of squares polynomials. Our other application is an $M$-SDP formulation of low rank tensor completion problems. These applications highlight that the $M$-SDP framework has theoretical and computational applications analogous to those of PSD matrices.  

\subsection{Application 1: $M$-SOS Polynomials}\label{sec:MSOS}

A well-known application of semidefinite programming is as a relaxation of polynomial optimization problems. 
The central idea is to search for nonnegativity certificates of a polynomial $f$ in the form of a decomposition of $f$ into a sum of squares (SOS); that is, if $\deg(f) = 2d$, one seeks degree-$d$ polynomials $g_1,\dots,g_k$ such that $f = \sum_{i=1}^k g_i^2$; see, e.g., \cite{lasserre_global_2001, parrilo_semidefinite_2003} for additional details. 
The maximum value of $\gamma$ such that $f(x)-\gamma$ is SOS is then a (possibly strict) lower bound on $\inf_{x \in \bbR^n} f(x)$. 
Additionally, a polynomial $f(x)-\gamma$ is SOS if and only if there is a positive semidefinite Gram matrix $Q$ such that $f(x) - \gamma = \xi^\top Q \xi$, where $\xi$ is a vector of all monomials of degree at most $d$. Searching for the maximum $\gamma$ such that $f(x)-\gamma$ is SOS is therefore a semidefinite programming problem. 
Here, we use $\starM$-PSD tensors to study a subset of SOS polynomials. 

SOS relaxations for structured polynomial optimization problems were studied using the tensor $t$-product in \cite{zheng_unconstrained_2022,marumo_t-semidefinite_2024} and we extend to the general $\starM$-product below. 
The group-theoretic interpretation developed in Section \ref{sec:Equivariance} provides further insight into the structure of SOS polynomials defined through $\starM$-products. 

It will be convenient to define a reshaping operator that partitions an $nn_3\times 1$ vector into $n$ blocks of length $n_3$ and converts the blocks into tubes of a lateral slice.  
More precisely, we define $\fold_{n_3}: \Rbb^{nn_3} \to \Rbb^{n\times 1\times n_3}$ such that if $v = \begin{bmatrix} v_1^\top & v_2^\top & \dots & v_{n}^\top \end{bmatrix}^\top$, where $v_i \in \bbR^{n_3}$ for each $i \in [n]$, then 
\begin{equation}\label{eq:fold_k}
\fold_{n_3}(v) = \begin{bmatrix}\tube(v_1) \\ \tube(v_2) \\ \vdots \\\tube(v_n)\end{bmatrix} \in \bbR^{n \times 1 \times n_3}.\end{equation}

We now define an M-SOS polynomial using the $\starM$-product and the folding operation in~\eqref{eq:fold_k}.  
\begin{defn}[$M$-SOS Polynomial, Gram Tensor]
Consider the ring of polynomials in $k$ variables with real coefficients and collect all possible monomials up to degree $d$ in the length $N = {k+d \choose d}$ vector $\xi = \begin{bmatrix} 1 & x_1 & x_2 & \dots &x_k & x_1^2 & x_1x_2 & \dots & x_k^d \end{bmatrix}^\top$. 
Let $N = n n_3$ for some positive integers $n$ and $n_3$ and let $\Xcal = \fold_{n_3}(\xi)$ be an $n\times 1\times n_3$ array of variables. 
 Given $M \in \Ogroup_{n_3}(\bbR)$, we say that a real polynomial $f$ in $k$ variables of degree at most $2d$ is \emph{$M$-SOS} if there is a tensor $\cQ \in \mathrm{PSD}_M^n$ such that 
    \begin{equation}f(x) = \langle \cX, \cQ \starM \cX\rangle.\end{equation}
Such a tensor $\cQ$ is called a \emph{Gram tensor} for $f$.
\end{defn}

We note that $M$-SOS polynomials are sums of squares polynomials. Indeed, if $\cQ$ is a Gram tensor for $f$, then by Proposition \ref{prop:square_root},  $\cQ = \cB^\top \starM \cB$ for some tensor $\cB \in \bbR^{r \times n \times n_3}$. It then follows that 
\begin{equation}
f(x) = \sum_{i =1}^{r} \langle B_{i,:,:} \starM \cX, B_{i,:,:} \starM \cX\rangle = \sum_{i =1}^n \left\|B_{i,:,:} \starM \cX\right\|^2,
\end{equation}
which is a sum of squares. Alternatively, the matrix representative of $\cQ$ is PSD by Proposition \ref{prop:MPSD_BlockMat}. 

\subsubsection{Invariant Polynomials} There has been recent interest in sums of squares polynomials which are invariant under a group action on the variables; see e.g., \cite{gatermann_symmetry_2004, riener_exploiting_2013, heaton_symmetry_2021}. In this setting, the theory of invariant SDPs is leveraged to better understand the SDPs which certify that an invariant SOS polynomial is a sum of squares. 

Since the results of Section \ref{sec:Equivariance} relate the $\starM$-product to invariant SDPs, we work towards an analogy in the tensor case. Specifically, Theorem \ref{thm:InvariantSDP} shows that if $\rho:G \to GL_{n_3}(\bbR)$ is a totally real representation, then an $M$-SDP with variables in $\mathrm{PSD}_M^m$ and equivariance constraints on tubes leads to an SDP invariant with respect to the representation on $\bbR^{mn_3}$ where $g \in G$ acts by $(I_m \otimes \rho(g))$. So, we expect that if $f$ is an $M$-SOS polynomial with a Gram tensor which has tubes representing $\rho$-equivariant transformations, then $f$ will be invariant under the action of $(I_m \otimes \rho(g))$ for $g \in G$. Unwinding the Kronecker product structure, this means that $f$ is invariant under the representation $\rho$ acting on \emph{groups} of variables of size $n_3$. While the existence of a Gram tensor satisfying equivariance constraints on tubes results in an invariant polynomial, the converse is not true without an additional assumption on the multiplicity of the irreducible representations appearing in the decomposition of $\bbR^{n_3}$.

\begin{restatable}{theorem}{QuadraticForms}\label{thm:QuadraticForms}
    Let $G$ be a finite group and $\rho:G \to GL_{n_3}(\bbR)$ be a totally real representation with each $\rho(g)$ orthogonal, and set $M \in \mathrm{O}_{n_3}(\bbR)$ and $W_\rho\subseteq \bbR_M$ to be as in the statement of Theorem \ref{thm:EquivariantSubspace}. 

    Consider a quadratic form in $mn_3$ variables $f \in \bbR[x_{1,1}, x_{1,2}, \ldots, x_{m,n_3}]_2$ and collect all variables in the vector $ \xi = \begin{bmatrix} x_{1,1} & x_{1,2} & \ldots & x_{1,n_3} & x_{2,1} & \dots & x_{m,n_3} \end{bmatrix}^\top$.
    
    If $f$ is $M$-SOS and there is a Gram tensor $\cQ$ for $f$ with tubes in $W_\rho$, then $f$ is SOS and invariant under the action of $G$ on $\bbR[x_{1,1},x_{1,2},\ldots, x_{m,n_3}]$ given by $g \cdot \xi = (I_m \otimes \rho(g))\xi$. The converse holds if the multiplicity of each irreducible representation in the decomposition of $\bbR^{n_3}$ appears with multiplicity at most one.  
\end{restatable}

\begin{proof}

As in the proof of Theorem \ref{thm:InvariantSDP}, we note that if $\cQ$ is a Gram tensor for $f$ with tubes in $W_\rho$ and 

\begin{align}
\hat{Q} = (I_m \otimes M^{\top}) \mymat_M(\cQ)(I_m \otimes M)
\end{align}
is the matrix representative of the underlying linear transformation, then $(I_m \otimes \rho(g))^\top \hat{Q}(I_m \otimes \rho(g)) = \hat{Q}$ for each $g \in G$. In particular, it follows that 

\begin{equation}(g\cdot f)(x) = ((I_m\otimes\rho(g))\xi)^\top \hat{Q} ((I_m\otimes\rho(g))\xi)  = \xi^\top \hat{Q}\xi = f(x).\end{equation}

Conversely, suppose that $f$ is SOS and invariant under the action of $G$. Then, there is a $(mn_3)\times (mn_3)$ positive semidefinite matrix $Q$ such that $f(x) = \xi^\top Q \xi$ and $(I_m \otimes \rho(g))^\top Q (I_m \otimes \rho(g)) = Q$. Denote blocks of $Q$ by 

\begin{equation}Q = \begin{bmatrix}Q_{1,1} & Q_{1,2} & \dots & Q_{1,m}\\ Q_{2,1} & Q_{2,2} & \dots & Q_{2,m}\\ \vdots & \vdots 
 & \ddots & \vdots \\ Q_{m,1} & Q_{m,2} & \dots & Q_{m,m}\end{bmatrix},\end{equation}
where each $Q_{i,j}$ is $n_3\times n_3$. Now, for each $i,j \in [m]$ and each $g \in G$, $\rho(g)^\top Q_{i,j}\rho(g) = Q_{i,j}$ since $(I_m \otimes \rho(g))^\top Q (I_m \otimes \rho(g)) = Q$. By hypothesis, the decomposition of $\bbR^{n_3}$ into irreducibles is given by 

\begin{equation}\bbR^{n_3} \simeq V_1 \oplus V_2 \oplus \dots \oplus V_s,\end{equation}
where each $V_i$ is unique. Since $M$ corresponds to a symmetry adapted basis of $\bbR^{n_3}$ it follows from Schur's Lemma (Lemma \ref{lem:Schur}) that 

\begin{equation}MQ_{i,j}M^{-1} = \begin{bmatrix} c_1I_{d_1} & & &\\ & c_2I_{d_2} & &\\ & &\ddots &\\ & & & c_sI_{d_s} \end{bmatrix}\end{equation}
where $d_t = \dim V_t$. By Theorem \ref{thm:EquivariantSubspace}, there exists a tube $\bq_{i,j} \in W_\rho$ such that $MQ_{i,j}M^{-1} = \diag(M \vec(\bq_{i,j}))$. So, there is a tensor $\cQ \in \bbR^{m \times m}_M$ with 

\begin{equation}Q = (I_m \otimes M^{\top}) \mymat_M(\cQ)(I_m \otimes M)\end{equation}
and such that $f = \langle \cX_1^{(n_3)}, \cQ \starM \cX_1^{(n_3)}\rangle$. 
\end{proof}

\subsubsection{Examples} We conclude this section with examples demonstrating Theorem \ref{thm:QuadraticForms} and its limitations. 

\begin{example}
 Let $S_3$ be the symmetric group on three elements and $\rho:S_3 \to GL_3(\bbR)$ be the permutation representation. Set
    \begin{center}
    
    \begin{equation}M = \begin{bmatrix}    -\frac{1}{\sqrt{3}} &    -\frac{1}{\sqrt{3}} &    -\frac{1}{\sqrt{3}}\\   \frac{1}{\sqrt{2}} &   0 &  \frac{-1}{\sqrt{2}}\\   \frac{1}{\sqrt{6}} &   -\frac{2}{\sqrt{6}} &   \frac{1}{\sqrt{6}}\end{bmatrix},  \enskip \bq_1 = (-\sqrt{3})\begin{bmatrix} 1\\ 1\\ 1\end{bmatrix}, \text{ and } \bq_2 = (\sqrt{18})\begin{bmatrix} -1\\ \sqrt{3}\\ 1\end{bmatrix}\end{equation}

    \end{center}

    The map $T_\ba:\bbR_M \to \bbR_M$ is $\rho$-equivariant if and only if $\ba \in \spann_\bbR\{\bq_1,\bq_2\}$. We construct an $M$-SOS quadratic form $f \in \bbR[x_1,x_2,x_3,y_1,y_2,y_3]$ as 
    
    \vspace{0.5cm}
    
        \begin{equation}\begin{aligned} f(x,y) &= \left\langle \begin{bmatrix} { x} \\ {  y} \end{bmatrix}, \left(\begin{bmatrix} \bq_1 + \bq_2\\ \bq_2 \end{bmatrix}\starM\begin{bmatrix} \bq_1 + \bq_2 & \bq_2 \end{bmatrix}\right) \starM \begin{bmatrix} { x}\\ {  y} \end{bmatrix} \right\rangle\\
        &= \left({ (-3 - \sqrt{2})x_1 + (3-\sqrt{2})x_2 + (3 - \sqrt{2})x_3} -{   (4 + \sqrt{2})y_1 + (2-\sqrt{2})y_2 + (2-\sqrt{2})y_3}\right)^2\\
        &+ \left({ (3 - \sqrt{2})x_1 + (-3 - \sqrt{2})x_2 + (3 - \sqrt{2})x_3} + {  (2-\sqrt{2})y_1  - (4 + \sqrt{2})y_2 + (2-\sqrt{2})y_3}\right)^2\\
        &+ \left({ (3 - \sqrt{2})x_1 + (3-\sqrt{2})x_2 + (-3 - \sqrt{2})x_3} + {   (2-\sqrt{2})y_1 + (2-\sqrt{2})y_2  - (4 + \sqrt{2})y_3} \right)^2
        \end{aligned}\end{equation}

        Note that $f$ is invariant under 

    \vspace{-0.5cm}
    
    \begin{equation}({     x_1,x_2,x_3}, {    y_1,y_2,y_3}) \mapsto ({     x_{g(1)},x_{g(2)},x_{g(3)}}, {    y_{g(1)},y_{g(2)}, y_{g(3)}}) \text{ for }  g \in S_3.\end{equation}
    
\end{example}

We also show that the condition that each irreducible representation appears in the decomposition of $\bbR^{n_3}$  with multiplicity at most one is necessary via an example.

\begin{example}\label{ex:HighMult_not_MSOS}
Let $S_2 = \{id,\sigma\}$, and $\rho:S_2\to GL_3(\bbR)$ be the representation with $\rho(\sigma) = \begin{bmatrix}1 & 0 &0\\ 0 & 0 & 1\\ 0 & 1 &0 \end{bmatrix}$. That is, $\sigma$ acts on vectors by swapping the second and third coordinates. Note that the decomposition of $\bbR^3$ into irreducible representations is given by

\begin{equation}\bbR^3 \simeq \spann \{(1,0,0)\} \oplus \spann \{(0,1,1)\} \oplus \spann\{(0,1,-1)\},\end{equation}
and that the trivial representation appears with multiplicity $2$, given by $\spann\{(1,0,0)\}$ and $\spann \{(0,1,1)\}$.  

Let $\alpha = \frac{1}{\sqrt{2}}$  and $M = \begin{bmatrix} 1 & 0& 0\\ 0 & \alpha & \alpha \\ 0 & \alpha & -\alpha\end{bmatrix}$. Note that an SOS quadratic form $f \in \bbR[x_1,x_2,x_3]_2$ is $M$-SOS if and only if a Gram matrix $Q$ for $f$ satisfies the condition that $MQM^{-1}$ is diagonal.

Let 

\begin{equation}f(x) = (x_1 + x_2 +x_3)^2 + (x_2 - x_3)^2 = x_1^2 + 2x_2^2 + 2x_3^2 + 2x_1x_2 + 2x_1x_3.\end{equation}
The first square $(x_1 + x_2 + x_3)^2$ is the square of the sum of elements lying in each copy of the trivial representation of $S_2$. The unique Gram matrix for $f$ is $Q = \begin{bmatrix}1 & 1 & 1\\ 1 & 2 & 0\\ 1 & 0 & 2 \end{bmatrix}$. Now, 

\begin{equation}MQM^{-1} = \begin{bmatrix}1 & \sqrt{2} & 0\\ \sqrt{2} & 2 & 0\\ 0 & 0 & 2\end{bmatrix}\end{equation}
is not diagonal. So, there is no tube $\bq \in \bbR^{1\times 1 \times 3}$ such that $Q = M^{-1}\diag(M\vec(\bq))M$ and therefore $f$ cannot be $M$-SOS. 
\end{example}

\subsection{Application 2: Low Rank Tensor Completion}\label{sec:TensorCompletion}

A common application of the $\starM$-product is low-rank completion of tensors by minimizing the tensor nuclear norm (e.g., \cite{kong_tensor_2021, zhang_corrected_2019, zhang_exact_2017, zhang_novel_2014}). In this setting, we consider a $\starM$-nuclear norm on $\bbR^{n_1\times n_2 \times n_3}$ which depends on the choice of orthogonal matrix $M$~\cite{song_robust_2020}: 
\begin{equation}\label{eq:M_Nuclear_Norm}
\|\cX\|_{M,*} = \sum_{k = 1}^{n_3} \|(\cX \times_3 M)_{:,:,k}\|_{*},
\end{equation}
where for a matrix $A$ with singular values $\sigma_1\geq \sigma_2 \geq \ldots \geq \sigma_r \geq 0$, $\|A\|_* = \sum_{j = 1}^r \sigma_j$ is the nuclear norm. It is known that the value of the matrix nuclear norm of $A$ can be computed via an SDP (see e.g., \cite{recht_guaranteed_2010} \cite[Section 2.1]{blekherman_semidefinite_2012}), with value given by the following primal-dual pair:

\begin{equation}\label{eq:NN_SDP}
\begin{aligned}
\max_{Y}\, & \tr(A^\top Y) & \qquad \qquad & \min_{W_1,W_2}\,  \frac{1}{2}\left(\tr W_1 + \tr W_2\right)\\
\text{s.t. } &\begin{bmatrix} I & Y \\ Y^\top & I\end{bmatrix}\succeq 0 & \qquad \qquad & \text{s.t. } \begin{bmatrix} W_1 & A\\ A^\top & W_2 \end{bmatrix} \succeq 0
\end{aligned}
\end{equation}

\subsubsection{$M$-SDP Formulation of \eqref{eq:M_Nuclear_Norm}} We show below that the SDP formulation of the matrix nuclear norm generalizes to compute \eqref{eq:M_Nuclear_Norm} using an $M$-SDP.
This will allow us to formulate the tensor completion problem as an $M$-SDP. 

\begin{proposition}\label{prop:MNN_as_p_SDP}
    Fix an $M \in \Ogroup_{n_3}(\bbR)$ and let $\cA\in \Rbb^{n_1\times n_2\times n_3}$. The $M$-nuclear norm $\|\cA\|_{M,*}$ can be computed by solving the following $M$-SDP
    \begin{equation}
\label{eq:NN_MSDP}\min_{\cW_1, \cW_2} \frac{1}{2}\left\langle \begin{bmatrix} \cI_M & 0\\ 0 & \cI_M \end{bmatrix}, \begin{bmatrix} \cW_1 & \cA\\ \cA^\top & \cW_2\end{bmatrix}\right\rangle  \quad  \subjectto \quad  \begin{bmatrix} \cW_1 & \cA\\ \cA^\top & \cW_2\end{bmatrix} \succeq_M 0
\end{equation}
for $\cW_1\in \Rbb^{n_1\times n_1\times n_3}$ and $\cW_2\in \Rbb^{n_2\times n_2\times n_3}$. 
\end{proposition}

\begin{proof}
Let $\widehat{\cA} = \cA\times_3 M$ and let $W_1^{(\ell)} \in \Rbb^{n_1\times n_1}$ and $W_2^{(\ell)} \in \Rbb^{n_2\times n_2}$ solve the independent matrix SDPs
    \begin{align}
        \min_{W_1,W_2}\, & \tfrac{1}{2}\left(\tr W_1 + \tr W_2\right) \quad \text{s.t.} \, \begin{bmatrix} W_1 & \widehat{\cA}_{:,:,\ell}\\ (\widehat{\cA}_{:,:,\ell})^\top & W_2 \end{bmatrix} \succeq 0
    \end{align}
for each $\ell\in [n_3]$. 
Define $\widehat{\cW}_1$ such that $(\widehat{\cW}_1)_{:,:,\ell} = W_1^{(\ell)}$ for $\ell\in [n_3]$ and let $\cW_1 = \widehat{\cW}_1 \times_3 M^\top$. 
Follow a similar procedure to construct $\cW_2$. 
It follows from~\Cref{prop:slices} that $\begin{bmatrix} \cW_1 & \cA\\ \cA^\top & \cW_2 \end{bmatrix}$ is $M$-PSD. 
Moreover, the corresponding objective value is 
\begin{subequations}\begin{align}
\frac{1}{2}\left\langle \begin{bmatrix} \cI_M & 0\\ 0 & \cI_M \end{bmatrix}, \begin{bmatrix} \cW_1 & \cA\\ \cA^\top & \cW_2\end{bmatrix}\right\rangle &= \frac{1}{2}\left\langle \begin{bmatrix} \cI_M \times_3 M & 0\\ 0 & \cI_M\times_3 M \end{bmatrix}, \begin{bmatrix} \cW_1 & \cA\\ \cA^\top & \cW_2\end{bmatrix} \times_3 M\right\rangle\\
&= \frac{1}{2}\sum_{\ell = 1}^{n_3}\left\langle \begin{bmatrix} I & 0\\ 0 & I\end{bmatrix}, \begin{bmatrix} W_1^{(\ell)} & \widehat{\cA}_{:,:,\ell}\\ (\widehat{\cA}_{:,:,\ell})^\top & W_2^{(\ell)} \end{bmatrix} \right\rangle\\
&= \frac{1}{2}\sum_{\ell = 1}^{n_3}(\tr W_1^{(\ell)} + \tr W_2^{(\ell)})\\
&= \sum_{\ell = 1}^{n_3} \|\widehat{\cA}_{:,:,\ell}\|_*\\
&= \|\cA\|_{M,*}.
\end{align} 
\end{subequations}

To show the $\starM$-nuclear norm is indeed the solution, consider  any tensors $\cW_1^{\prime} \in \Rbb^{n_1\times n_1\times n_3}$ and $\cW_2^{\prime} \in \Rbb^{n_2\times n_2 \times n_3}$ such that $\begin{bmatrix} \cW_1^{\prime} & \cA\\ \cA^\top & \cW_2^{\prime}\end{bmatrix}$ is M-PSD. 
By Proposition \ref{prop:slices}, any $M$-PSD tensor with block format yields feasible solutions to \eqref{eq:NN_SDP} by taking frontal slices in the transform domain. Moreover, the corresponding objective value satisfies
\begin{subequations}\begin{align}\frac{1}{2}\left\langle \begin{bmatrix} \cI_M & 0\\ 0 & \cI_M \end{bmatrix}, \begin{bmatrix} \cW_1^\prime  & \cA\\ \cA^\top & \cW_2^\prime \end{bmatrix}\right\rangle 
&= \frac{1}{2} \sum_{\ell = 1}^{n_3} (\tr (\cW_1^\prime \times_3 M)_{:,:,\ell} + \tr(\cW_2^\prime \times_3 M)_{:,:,\ell})\\
&\geq \sum_{\ell = 1}^{n_3} \|(\cA \times_3 M)_{:,:,\ell}\|_*\\
&= \|\cA\|_{M,*}.
\end{align}\end{subequations}

Thus, solving \eqref{eq:NN_MSDP} computes $\|\cA\|_{M,*}$. 
\end{proof}

\subsubsection{Tensor Completion} We now consider the tensor completion problem. A tensor $\cY \in \bbR^{n_1 \times n_2 \times n_3}$ is \emph{partially specified} if the entries $\cY_{i,j,k}$ are known for $(i,j,k) \in \Omega \subseteq \bbN^3$ and unknown $\cY_{i,j,k}$ for $(i,j,k)\not \in \Omega$. We seek a low $M$-rank solution to the tensor completion problem, using the $M$-nuclear norm \eqref{eq:M_Nuclear_Norm} as a proxy for rank in the objective function:

\begin{equation}\label{eq:MREC}
\min_{\cA} \|\cA\|_{M,*} \, \text{ s.t. } \, \cA_{i,j,k} = \cY_{i,j,k} \text{ for all } (i,j,k) \in \Omega.
\end{equation}

We rewrite \eqref{eq:MREC} as an $M$-SDP:

\begin{equation}\label{eq:Fixed_M_NNSDP}
\min_{\cA, \cW_1, \cW_2} \frac{1}{2}\left\langle \begin{bmatrix} \cI_M & 0 \\ 0 & \cI_M\end{bmatrix}, \begin{bmatrix} \cW_1 & \cA\\ \cA^\top & \cW_2\end{bmatrix}\right\rangle \quad \text{s.t.} \quad \cA_{i,j,k} = \cY_{i,j,k} \text{ for all } (i,j,k)\in \Omega, \, \begin{bmatrix} \cW_1 & \cA\\ \cA^\top & \cW_2\end{bmatrix} \succeq_M 0. 
\end{equation}
Note that if $\cY \in \bbR^{n_1\times n_2 \times n_3}$, then the tensors involved in \eqref{eq:Fixed_M_NNSDP} have sizes $\cW_1 \in \bbR^{n_1 \times n_1 \times n_3}, \cA \in \bbR^{n_1 \times n_2 \times n_3}, $ and $\cW_2 \in \bbR^{n_2 \times n_2 \times n_3}$.

\section{Numerical Results for Tensor Completion via $M$-SDP}\label{sec:Completion_Experiments}

We present a tensor completion under the $\starM$-product for two different datasets. 
We leverage the Python package \texttt{cvxpy} \cite{diamond_cvxpy_2016,agrawal_rewriting_2018,agrawal_disciplined_2019} to solve $M$-SDP. 
We use the default algorithm called the splitting conic solver (SCS)~\cite{odonoghue_conic_2016}, which uses first-order derivative information via operator splitting (\'{a} la alternating direction method of multipliers~\cite{boyd_distributed_2010}) to scalably solve SDPs. 
All code and reproducible experiments are available at \url{https://github.com/elizabethnewman/tsdp}.

\SetKwComment{Comment}{\% }{ }

\begin{algorithm}[t]
\caption{$M$-SDP for Tubally Sparse Tensor Completion}\label{alg:msdp_tensor_completion}
\KwData{partially-observed tensor $\cY \in \Rbb^{n_1\times n_2\times n_3}$,  
    observed tubal entries indices $\Omega \subset [n_1] \times [n_2]$, 
    orthogonal $M\in \Rbb^{n_3}$}
\KwResult{completed tensor $\cA\in \Rbb^{n_1\times n_2\times n_3}$ with $\cA_{i,j,:} = \cY_{i,j,:}$ for $(i,j)\in \Omega$}

$\widehat{\cY} = \cY \times_3 M$ \Comment*[r]{move to transform domain (only multiply along observed tubes)}
  
\For{$k = 1,2,\ldots n_3$}{

Solve the $(n_1 + n_2) \times (n_1 + n_2)$ SDP  
\begin{align*}
    \min_{W_1,W_2,X} \tr(W_1) + \tr(W_2)
    \quad \text{s.t.} 
    \quad \begin{bmatrix}
    W_1 & X\\
    X^\top & W_2
\end{bmatrix} \succeq 0 \text{ and } X_{i,j} = \widehat{\cY}_{i,j,k} \text{ for } (i,j)\in \Omega
\end{align*}

$\widehat{\cA}_{:,:,k} \gets X$\;
}
$\cA = \widehat{\cA} \times_3 M^\top$ \Comment*[r]{return to spatial domain}
\end{algorithm}

\subsubsection{\bf Parameters} We compare the performance of $M$-SDP for various choices of orthogonal transformation, including the identity matrix, $I$, the discrete cosine transform matrix, $C$, the Haar wavelet matrix, $H$, and the (transposed) left-singular matrix of the mode-$3$ unfolding of $\cY$, $U_3^\top$ (for details, see~\cite{kilmer_tensor-tensor_2021}), and a random orthogonal matrix $Q$. 
We will compare to an equivalent matrix SDP formulation where we first unfold the tensor into a matrix $A\in \Rbb^{n_1n_3\times n_2}$ by concatenating the frontal slices of the tensor vertically; that is, $Y = \begin{bmatrix} \cY_{:,:,1}^\top & \cdots & \cY_{:,:,n_3}^\top \end{bmatrix}^\top$.  

\subsubsection{\bf Metrics} 
We measure performance via relative error on the fully-observed dataset using the max-norm; specifically, 
    \begin{align}
        \frac{\|\cA - \cY\|_{\max}}{\|\cY\|_{\max}}
    \end{align}
where $\|\cB \|_{\max} = \max_{i,j,k} |\cB_{i,j,k}|$.

\subsubsection{\bf Computational Cost} We summarize the $M$-SDP pipeline in~\Cref{alg:msdp_tensor_completion}. 
Computationally, an $M$-SDP requires solving $n_3$ independent  $(n_1 + n_2) \times (n_1 + n_2)$ matrix SDPs; thus, these small SDPs can be solved in parallel. 
In comparison, the comparable matrix SDP formulation requires a single solve of an $(n_1n_3 + n_2) \times (n_1n_3 + n_2)$ matrix SDP. 
The computational cost of solving an SDP is dominated by the projection onto the convex cone, which requires solving a linear system.

\subsection{Video Data}\label{sec:video_data}

We solve the tensor completion problem~\eqref{eq:Fixed_M_NNSDP} with the classical escalator video\footnote{The link to download a {\sc Matlab} file containing the escalator data can be found at \url{https://cvxr.com/tfocs/demos/rpca/}.}. 
The data consists of a surveillance video of pedestrians taking an escalator recorded with a stationary camera. 
The original data is consists of video frames of size $130 \times 160$ with $200$ frames.  
For computational efficiency using the basic SDP algorithms built into \texttt{cvxpy}, we crop and subsample the frames. 
We create two datasets of different sizes. 
For the small dataset, we create a tensor $\cY$ of size $32 \times 32 \times 8$, corresponding to $\text{height} \times \text{width}\times \text{frames}$.   
We store every sixth frame in the first $48$ frames. 
For the large dataset, we create a tensor $\cY$ of size $64 \times 64 \times 128$, storing the first $128$ consecutive frames. 
In both cases, we crop the image in the center to capture foreground activity (i.e., frames include pedestrians, not just background). 
We normalize the data via $\cY / \|\cY\|_F$.  
To form a partially-observed dataset, we randomly generate a single set of indices $\Omega \in [\text{height}] \times [\text{width}]$ and retain only $25\%$ of the original data points; i.e., $|\Omega| / (\text{height}\times \text{width}) = 0.25$. 
We observe the same set of indices for each video frame, resulting in a tubally-sparse representation of the video data. 
We report our results on the small dataset in~\Cref{fig:escalator_approximation_small} and on the large dataset in~\Cref{fig:escalator_approximation_large}.

\begin{figure}
    \centering
    \subfloat[Approximations for tensor and matrix completion. 
    The top row show the frames of the cropped, subsampled data $\cY_{:,:,j}$. 
    The second row shows the masked frames, using the same mask for each frame, where dark colors indicate unobserved entries. 
    The next five rows show the completed data using $M$-SDP for various choices of fixed transformation $M$. 
    The last row shows the matrix completion comparison. 
    All images in a given column are displayed on the same colorscale. \label{fig:video_small_approx}]
    {\begin{tikzpicture}
    \draw[opacity=0] (-0.5\linewidth,0) -- (0.5\linewidth,0);
    \node {\includegraphics[width=0.8\linewidth]{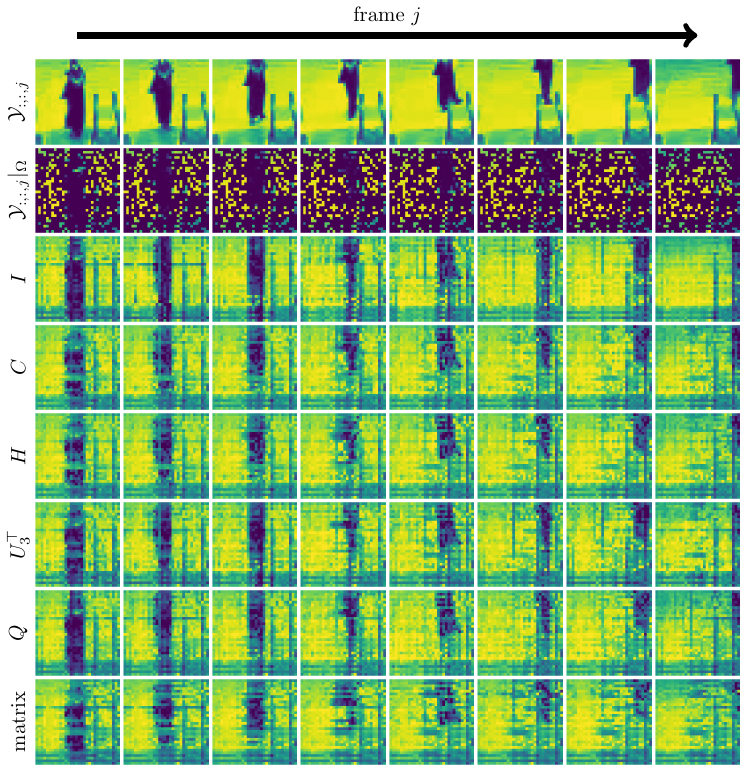}};
    \end{tikzpicture}}

    \subfloat[Relative error for completion of escalator video. 
    We report the worst-case relative error per frame using the max-norm, which returns the largest error over all possible pixels. 
     \label{fig:video_small_error}]
    {\begin{tikzpicture}
    \draw[opacity=0] (-0.5\linewidth,0) -- (0.5\linewidth,0);
    \node {\includegraphics[width=0.7\linewidth]{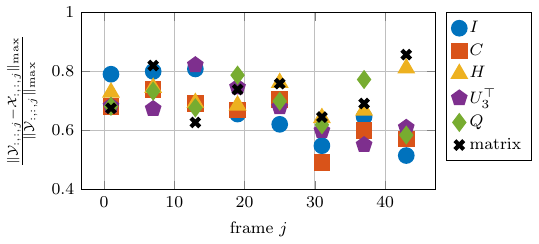}};
    \end{tikzpicture}}

    \caption{Tensor and matrix completion for small video data set $\cY$ of size $32\times 32\times 8$.}
    \label{fig:escalator_approximation_small}
\end{figure}

In the small-scale data experiment in~\Cref{fig:escalator_approximation_small}, we observe that the choice of $M$ does not significantly impact performance. 
We further note that the tensor completion strategies often performed slightly better than matrix completion when examining the error per frame in~\Cref{fig:video_small_error}. 
We compare the max-norm relative errors instead of the Frobenius norm because all methods performed similarly under the Frobenius norm. 
The similar performance for both the matrix and tensor completion indicates that exploiting multilinear structure does not sacrifice quality of approximation and can bring new computational advantages.

\begin{figure}
    \centering
    \subfloat[Approximations for tensor completion. 
    The top row shows every eighth frame in the first $64$ frames of the data. 
    Subsequent frames contain only background information. 
    The second row shows the masked frames, using the same mask for each frame where dark colors indicate unobserved entries. 
    The last five rows show the completed data for various choices of fixed transformation $M$. 
    Images in a particular column are displayed using the same colorscale. \label{fig:video_large_approx}]
    {\begin{tikzpicture}
    \draw[opacity=0] (-0.5\linewidth,0) -- (0.5\linewidth,0);
    \node {\includegraphics[width=0.8\linewidth]{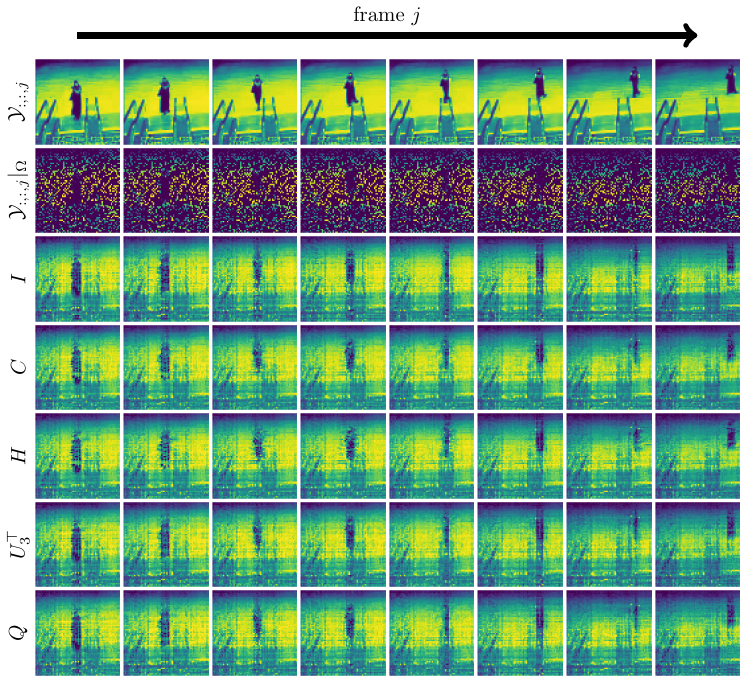}};
    \end{tikzpicture}}

    \subfloat[Relative error for completion of escalator video. 
    We display the worst-case relative error per frame using the max-norm, which returns the largest error over all possible pixels. \label{fig:video_large_error}]
    {\begin{tikzpicture}
    \draw[opacity=0] (-0.5\linewidth,0) -- (0.5\linewidth,0);
    \node {\includegraphics[width=1\linewidth]{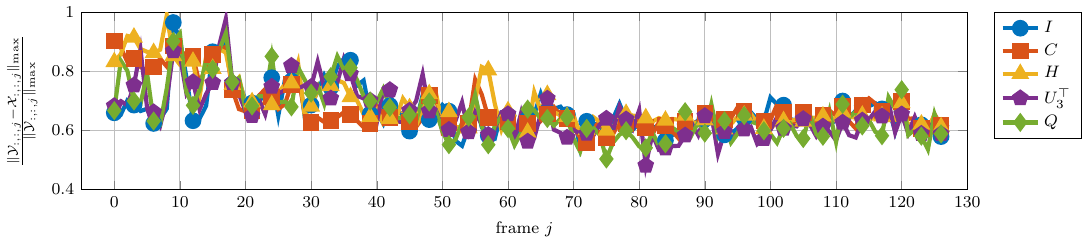}};
    \end{tikzpicture}}
    
    \caption{Tensor completion for large video data set $\cY$ of size $64\times 64\times 128$.}
    \label{fig:escalator_approximation_large}
\end{figure}

For the large-scale data experiment in~\Cref{fig:escalator_approximation_large}, we again observe that the choice of $M$ does not significantly impact performance for video completion. 
We do not compare to the matrix case because the basic \texttt{cvxpy} implementations cannot support such large data structures. 
This is a key computational advantage of the $M$-SDP framework. 
After applying the transformation, we effectively solve a small-scale matrix completion problems per frontal slice. 
This both significantly reduces the size of the problem and enables the problems to be solved in parallel. 

\subsection{Hyperspectral Data}\label{sec:hyperspectral}

We solve the tensor completion problem~\eqref{eq:Fixed_M_NNSDP} for the Indian Pines hyperspectral dataset~\cite{baumgardner_220_2015}, originally downloaded from {\sc Matlab}'s Hyperspectral Imaging Toolbox\footnote{The dataset be can loaded in {\sc Matlab}  using \texttt{hcube = hypercube('indian\_pines.dat'); A = hcube.DataCube;}}. 
Hyperspectral imaging measures the energy of a region for various light spectrum wavelengths. 
The absorption and reflection of each wavelength gives insight into the material. 
The Indian Pines dataset is a tensor $\cY$ of size $145 \times 145\times 220$, corresponding to $\text{height} \times \text{width} \times \text{wavelength}$. 
Each $\text{height} \times \text{width}$ frame is a birds eye view of farmland. 
We normalize the data via $\cY / \|\cY\|_F$.  
We again randomly generate a single $\text{height}\times \text{width}$ mask, retaining only $25\%$ of the original data points. 
We report our results in~\Cref{fig:hyperspectral_data}.

\begin{figure}
    \centering
    \subfloat[Approximations for tensor completion. 
    The top row shows the original frames, the second row shows the masked frames, using the same mask for each frame where dark colors indicate unobserved entries, and the last four rows show the completed data for various choices of fixed transformation $M$. 
    The first column to shows an RGB rendering of the image using bands (R,G,B) = (26, 16, 8). 
    The remaining columns show a false color scale rendering of a particular frame of the hyperspectral data, chosen to highlight similarities and differences in relative error. 
    Each false color scale is the same for a given column. \label{fig:hyperspectral_approx}]
    {\begin{tikzpicture}
    \draw[opacity=0] (-0.5\linewidth,0) -- (0.5\linewidth,0);
    \node {\includegraphics[width=1\linewidth]{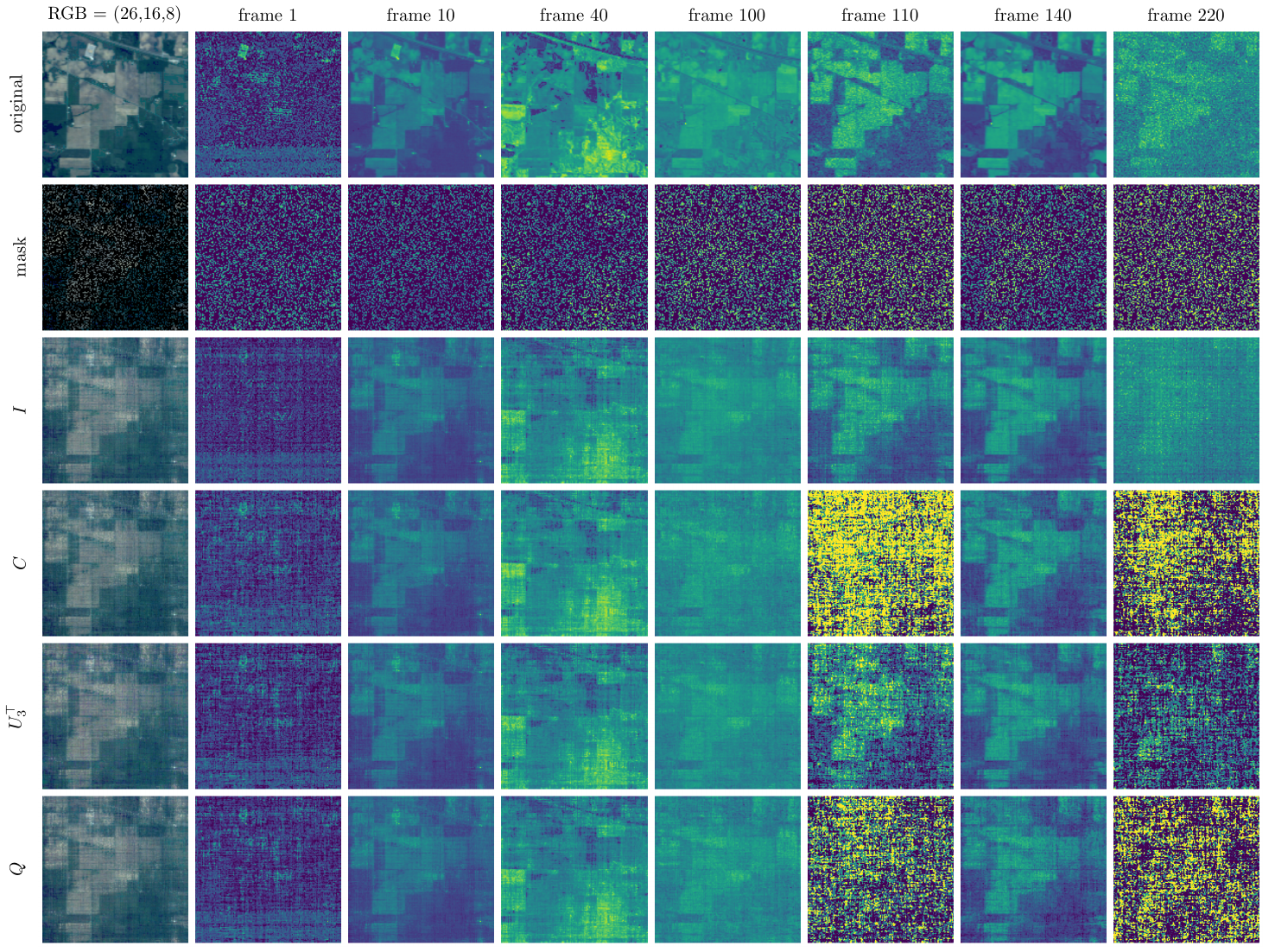}};
    \end{tikzpicture}}

    \subfloat[Relative error for completion of hyperspectral imaging data. 
    We display the worst-case relative error per frame using the max-norm, which returns the largest error over all possible pixels. 
    We see similar patterns for the Frobenius norm error. \label{fig:hyperspectral_error}]
    {\begin{tikzpicture}
    \draw[opacity=0] (-0.5\linewidth,0) -- (0.5\linewidth,0);
    \node {\includegraphics[width=1\linewidth]{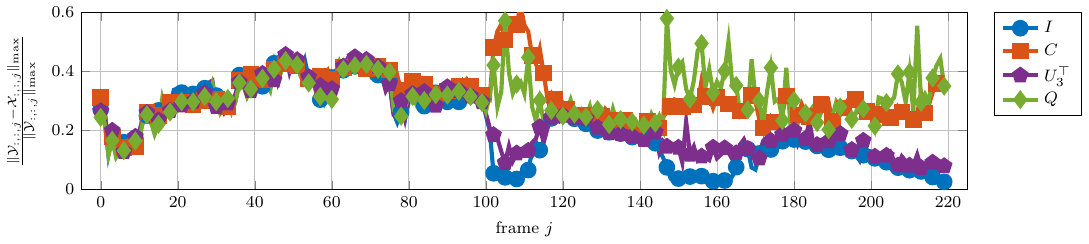}};
    \end{tikzpicture}}
    
    \caption{Tensor completion for hyperspectral data set $\cY$ of size $145\times 145\times 220$.}
    \label{fig:hyperspectral_data}
\end{figure}

In~\Cref{fig:hyperspectral_approx}, we examine the tensor completion of the hyperspectral data. 
While the RGB images formed from the various transformations appear similar, the per-frame approximations demonstrate that the identity matrix $I$ qualitatively produces the closest approximation to the original frame, followed by the data-dependent transformation $U_3^\top$. 
In~\Cref{fig:hyperspectral_error}, we see that for the first approximately $100$ frames, all methods produce similar errors. 
Then, we start to see a gap where the DCT matrix $C$ and random matrix $Q$ perform significantly worse than the identity $I$ and data-dependent $U_3^\top$ cases. 

\section{Conclusions and Future Work}

In this paper, we developed the theory of $\starM$-PSD tensors and the corresponding semidefinite programming problems. Additionally, we connected the $\starM$-product of third order tensors to the representation theory of finite groups, explicitly relating the choice of transformation matrix $M$ and underlying group action. This connection allowed us to interpret $M$-semidefinite programming problems in terms of group invariant SDP. The framework of $\starM$-semidefinite programming has applications which are analogous to those of standard semidefinite programming, and we use the framework to describe sums of squares quadratic forms and tensor completion problems. Proof-of-concept computations involving tensor completion problems suggest that the $M$-SDP framework is able to scale to larger problems than a na\"ive matrix SDP approach and that some problems have a non-trivial dependence on the choice of transformation matrix $M$. 

In our computational experiments, we used an out of the box solver in \texttt{cvxpy} \cite{diamond_cvxpy_2016}. Future work could develop specialized numerical methods for $M$-SDP which take advantage of tensor structure. Additionally, an interesting avenue to pursue is a bilevel optimization framework for problems such as tensor completion, where the inner level is an $M$-SDP and the outer level is optimization over the transformation matrix $M$. 

\section{Acknowledgements}

The work by E. Newman was partially supported by the National Science Foundation
(NSF) under grants [DMS-2309751] and [DE-NA0003525].

Sandia National Laboratories is a multimission laboratory managed and operated by National Technology \& Engineering Solutions of Sandia, LLC, a wholly owned subsidiary of Honeywell International Inc., for the U.S. Department of Energy’s National Nuclear Security Administration under contract DE-NA0003525.
This paper describes objective technical results and analysis. Any subjective views or opinions that might be expressed in the paper do not necessarily represent the views of the U.S. Department of Energy or the United States Government.

\appendix

\section{Constructing Transformation Matrices}

The choices of transformation matrix $M$ used in our numerical experiments in Section \ref{sec:Completion_Experiments} are summarized in Table \ref{tab:transformations}.

\begin{table}
\centering
\caption{Transformations}
\label{tab:transformations}

\scriptsize
\begin{tabular}{|lcl|}
\hline
Name & Notation & Computation\\
\hline\hline
Identity & $I$ & \texttt{I = numpy.eye(n3)} \\
\hline
Discrete Cosine Transform & $C$ & \texttt{C = scipy.fftpack.dct(n3)} \\
\hline
Left-Singular Matrix & $U_3^\top$ & \texttt{U3, \_, \_ = np.linalg.svd(mode3unfold(Y), full\_matrices=False)}\\
\hline
Haar Wavelet & $H$ & adapted from \href{https://stackoverflow.com/questions/23869694/create-nxn-haar-matrix}{StackOverflow} \\
\hline
Random Orthogonal & $Q$ & \texttt{Q, \_ = numpy.linalg.qr(numpy.random.randn(n))} \\
\hline

\end{tabular}

\end{table}

\section{Proof of Lemma \ref{lem:Trace_transpose}}\label{sec:Appendix_Proof}

\begin{proof}
The proofs are calculations. Set $\cX_{i,1,:} = \bx_i$. 

For the first claim, set $\cA = \begin{bmatrix} \ba_1 & \ba_2 & \ldots & \ba_n\end{bmatrix}$ and $\cB = \begin{bmatrix} \bb_1 & \bb_2 & \cdots & \bb_n\end{bmatrix}$. Then,

\begin{subequations}\begin{align}
\langle \cA \starM \cX, \cB \starM \cX \rangle &= \left\langle \sum_{i = 1}^{n} \ba_{i} \starM \bx_i, \sum_{j = 1}^{n} \bb_j \starM \bx_i \right \rangle\\
&= \left\langle \sum_{i = 1}^{n} M^\top \diag(M\vec(\ba_{i}))M\vec(\bx_i), \sum_{j = 1}^{n} M^\top \diag(M\vec(\bb_{j}))M\vec(\bx_j) \right \rangle\\
&= \left(M^\top \sum_{i = 1}^{n}\diag(M\vec(\ba_{i}))M\vec(\bx_i)\right)^\top \left(M^\top \sum_{j = 1}^{n}\diag(M\vec(\bb_{j}))M\vec(\bx_j)\right)\\
&= \sum_{i =1}^n \vec(\bx_i)^\top \left(\sum_{j = 1}^{n}M^\top\diag(M\vec(\ba_{i}))\diag(M\vec(\bb_{j}))M\vec(\bx_j) \right)\\
&= \sum_{i =1}^n \vec(\bx_i)^\top \left(\sum_{j = 1}^{n}M^\top\diag\left[\diag(M\vec(\ba_{i}))M\vec(\bb_{j})\right]M\vec(\bx_j)\right)\\
&= \langle \cX, (\cA^\top \starM \cB)\starM \cX \rangle
\end{align}\end{subequations}

For the second claim, let $\cA_{i,j,:} = \ba_{i,j}$. Then,

\begin{subequations}\begin{align}
\langle \cX, \cA \starM \cX \rangle &= \sum_{i = 1}^n \vec(\bx_i)^\top \left(\sum_{j = 1}^n M^\top\diag(M\ba_{i,j})M\bx_j\right)\\
&= \sum_{i = 1}^n \sum_{j = 1}^n \vec(\bx_i)^\top M^\top \diag(M\vec(\bx_j))M\vec(\ba_{i,j})\\
&= \sum_{i = 1}^n\sum_{j = 1}^n \left(M^\top \diag(M \vec(\bx_i)M\vec(\bx_j))\right)^\top \vec(\ba_{i,j})\\
&= \langle \cA, \cX \starM \cX^\top \rangle.
\end{align}\end{subequations}

For the third claim, we compute

\[\langle \cA \times_3 M, \cB \times_3 M \rangle = \sum_{i = 1}^{n_1} \sum_{j = 1}^{n_2} (M\vec(\ba_{i,j})^\top (M\vec(\bb_{i,j})) = \sum_{i = 1}^{n_1}\sum_{j = 1}^{n_2} \vec(\ba_{i,j})^\top \vec(\bb_{i,j}) =\langle \cA, \cB \rangle.\]
\end{proof} 

\printbibliography
\end{document}